\documentclass[10pt, a4paper]{amsart}
\usepackage{geometry}                
\usepackage{graphicx}
\usepackage{epstopdf}
\usepackage{amsmath}
\usepackage{amsthm}
\usepackage{color}
\usepackage{amscd}
\usepackage{amsfonts}
\usepackage{epsfig}
\usepackage{amssymb,latexsym}
\usepackage{bm}

\def\e{\epsilon}
\def\espacio{H^2([0, 1]; \R^2)}
\def\espaciou{H^1([0, 1]; \R^2)}
\def\espacioo{H^2_O([0, 1]; \R^2)}

\def\a{\alpha}
\def\bbF{\mathbb F}
\def\bJ{\mathbf J}
\def\C{\mathcal C}
\def\bbO{\mathbb O}
\def\bbI{\mathbb I}
\def\bbU{\mathbb U}
\def\bbH{\mathbb H}

\def\bW{\mathbf{W}}
\def\X{\mathbb{X}}
\def\Y{\mathbb{Y}}
\def\V{\mathbb{V}}

\def\R{\mathbb{R}}
\def\bR{\mathbf{R}}

\def\X{\mathbb{X}}
\def\F{\mathbb{F}}

\def\N{\mathbb{N}}

\def\H{\mathbb{H}}
\def\cC{\mathcal{C}}

\def\G{\mathbb{G}}

\def\bP{\mathbf{P}}
\def\bP{\mathbf{P}}
\def\bbP{\mathbb{P}}
\def\B{\mathbb{B}}
\def\bF{\mathbf{F}}
\def\bG{\mathbf{G}}
\def\bK{\mathbf{K}}

\def\L{\mathbb{L}}

\def\bX{\mathbf{X}}

\def\by{\mathbf{y}}
\def\bz{\mathbf{z}}
\def\bQ{\mathbf{Q}}

\def\bbS{\mathbb{S}}
\def\bbC{\mathbb{C}}

\def\bn{\mathbf{n}}

\def\bu{\mathbf{u}}
\def\bU{\mathbf{U}}

\def\bv{\mathbf{v}}
\def\bV{\mathbf{V}}
\def\bw{\mathbf{w}}

\def\br{\mathbf{r}}
\def\bz{\mathbf{z}}
\def\bZ{\mathbf{Z}}
\def\bc{\mathbf{0}}
\def\bcero{\mathbf{0}}
\def\id{\mathbf{1}}

\newcommand{\dv}{\operatorname{div}}

\newcommand{\supp}{\operatorname{supp}}

\def\bp{\mathbf{p}}
\def\bq{\mathbf{q}}
\def\bH{\mathbf{H}}

\newtheorem{theorem}{Theorem}[section]
\newtheorem{corollary}[theorem]{Corollary}
\newtheorem{definition}{Definition}[section]
\newtheorem{lemma}[theorem]{Lemma} 
\newtheorem{proposition}[theorem]{Proposition}

\newcommand{\Dv}{\operatorname{Div}}

\numberwithin{equation}{section}

\title[Morse theory]{Hilbert's 16th Problem through Morse Theory}
\author{Pablo Pedregal}
\address{Universidad de Castilla-La Mancha}
\curraddr{}
\email{pablo.pedregal@uclm.es}
\thanks{Departamento de Matemáticas, Universidad de Castilla-La Mancha, 13071 Ciudad Real, SPAIN. Supported by grants PID2023-151823NB-I00, and  SBPLY/23/180225/000023}


\date{}

\begin{document}
\maketitle
\begin{abstract}
We focus on the second part of Hilbert's 16th problem and provide an upper bound on the number of limit cycles that a polynomial, differential, planar system may have, depending exclusively on the degree $n$ of the system. Such a bound turns out to be a polynomial of degree $4$ in $n$. More specifically, if $H(n)$ indicates the maximum number of limit cycles among planar, differential, polynomial systems of degree $n$, then
\begin{gather}
H(n)\le \dfrac52 n^4-\dfrac{23}2 n^3+ \dfrac{43}2n^2-\dfrac{37}2n+7\,\,\,\,
\mbox{if $n$ is even, and} \nonumber\\
 H(n)\le \dfrac52 n^4-\dfrac{23}2 n^3+ \dfrac{41}2n^2-\dfrac{33}2n+6\,\,\,\,
\mbox{if $n$ is odd}.\nonumber
\end{gather}
For quadratic systems, we find $H(2)=4$. Our proof is entirely variational and utilizes in a fundamental way tools and facts from Morse theory, to the point that no particular expertise in dynamical systems is necessary or required. 

It is particularly interesting to compare our results with the ones for the same problem recently announced and made public by OpenAI as part of the solution of hundreds of important mathematical problems. 
\end{abstract}
\tableofcontents

\section{Introduction and main result}
This paper deals with planar differential systems of the form
\begin{equation}\label{e1}
x'(t)=P(x(t), y(t)),\quad y'(t)=Q(x(t), y(t)).
\end{equation}
As soon as the two functions
$P(x, y), Q(x, y):\R^2\to\R$
are smooth, there is a unique smooth integral curve of the system passing through every point in $\R^2$. This is a standard result that is studied in every first course in ODEs. Equilibria, i.e. points $(x_0, y_0)\in\R^2$ where both $P$ and $Q$ simultaneously vanish, play a central role in the overall dynamics of such a system. Periodic solutions too deserve a special place concerning the global dynamics of the system.
\begin{definition}
Every periodic solution of \eqref{e1} that is isolated from other such solutions is called a limit cycle of the system. 
\end{definition} 
We restrict attention in this paper to the case of planar, polynomial, differential systems for which the maximum of the degrees of 
both polynomials $P(x, y)$ and $Q(x, y)$ is $n$. 

The problem we would like to address is the second part of Hilbert's 16th problem 
\cite{Hi}, which in the version due to Smale \cite{Sm}, can be formulated as follows.
\begin{quote}
Consider the polynomial differential system \eqref{e1} in
$\R^2$. Is there a bound K on the number of limit cycles of the form
$K\leq n^q$ where $n$ is the maximum of the degrees of $P$ and $Q$,
and $q$ is a universal constant\,?
\end{quote}

There have been some crucial contributions towards the solution of this problem. Some of them are indicated below at the end of this Introduction. 
As our approach has essentially nothing in common with those, we simply mention some additional ones here without further comment: \cite{Bam}, \cite{CW}, \cite{Du}, \cite{Ec}, \cite{I0}, \cite{Il}, \cite{PL1}, \cite{PL2}, \cite{Po1}, \cite{Sh}, \cite{So}. There are also some important articles dealing with lower bounds for the number of limit cycles \cite{CL1}, \cite{HL}, \cite{Li}, as well as some other relevant results concerning this problem restricted to algebraic limit cycles \cite{Lli}, \cite{LRS},  \cite{LRS2}, \cite{XZ} among others; or about the possible configurations of those limit cycles \cite{GLV}, \cite{LR}, \cite{SS}, \cite{Sv}, \cite{Wi}. There is such an abundance of relevant papers, that, to avoid dispersion, we have only mentioned those more familiar to the author in the final bibliographic section. 

Our aim is the complete proof of the following central theorem.

\begin{theorem}\label{principall}
Consider the polynomial differential system \eqref{e1} of degree $n>1$. Assume that:
\begin{enumerate}
\item $P$ and $Q$ have no common non-trivial factor.
\item All of the connected components of the algebraic curve 
$P_x+Q_y=0$
are homeomorphic to a straight line or to an oval (i.e. such
curve has no singular points), and that there are 
$M$ of such components.
\item The polynomial system
\begin{equation}\label{salto0}
\begin{array}{r}
P(P_{xx}+Q_{yx})+Q(P_{xy}+Q_{yy})=0,\\
P_x+Q_y=0,
\end{array}
\end{equation}
only has $N$ simple solutions
(i.e. $N$ simple contact points of the vector field $(P,Q)$ with the curve
$P_x+Q_y=0$).
\end{enumerate}
Then an upper bound for the number $H(n)$ of limit cycles  that such a 
differential system \eqref{e1} may have is
\begin{equation}\label{cotaprin}
H(n)\le 1+(n-1)^2(M+N).
\end{equation}
\end{theorem}
Though this theorem can be used to find upper bounds for the number of limit cycles of special families of differential systems, or even for individual ones, 
upper bound \eqref{cotaprin} can serve, together with the classic results of Bezout and Harnack to relate parameters $M$ and $N$ to $n$, to establish an upper bound for the number of limit cycles of such regular, planar, polynomial, differential system of degree $n$ only in terms of $n$. 
On the other hand, the important hypotheses assumed in this statement are generic, in the sense that if they do not hold for a particular planar polynomial system, a small perturbation of $P$ and $Q$, without changing their degree, permits to have them.  A standard genericity argument then may be utilized to extend bound \eqref{cotaprin} in Theorem \ref{principall} to arbitrary polynomial systems of degree $n$, not complying with such conditions, in order to show the validity of the following general explicit upper bound furnishing an answer to Hilbert's 16th problem (\cite{Hi}) and Smale's 13th problem (\cite{Sm}). 

\begin{theorem}\label{t5}
An upper bound for the maximum number $H(n)$ of limit cycles that a planar
polynomial differential system of degree $n>1$ can have is
\begin{gather}
H(n)\le \dfrac52 n^4-\dfrac{23}2 n^3+ \dfrac{43}2n^2-\dfrac{37}2n+7\,\,\,\,
\mbox{if $n$ is even, and} \nonumber\\
 H(n)\le \dfrac52 n^4-\dfrac{23}2 n^3+ \dfrac{41}2n^2-\dfrac{33}2n+6\,\,\,\,
\mbox{if $n$ is odd}.\nonumber
\end{gather}
\end{theorem}
The number $H(n)$ is usually called the {\it Hilbert number} for 
polynomial differential systems of degree $n$.
This upper bound for $H(n)$ yields a universal exponent $q=4$ for the above version due to Smale. 

Part of our job is to
understand the role played by the following three pieces of
information:
\begin{itemize}
\item the divergence curve 
$$
\Dv=\Dv(x, y)\equiv P_x(x, y)+Q_y(x, y)=0;
$$ 
\item the set of contact points of the vector field $(P,Q)$ with the curve
$\Dv=0$ which are the solutions of system \eqref{salto0};
and
\item the role played by the factor $(n-1)^2$ multiplying the sum $M+N$ in \eqref{cotaprin}, and where it comes from.
\end{itemize}

\subsection{Main corollaries of central theorem}\label{hilb}
It is a well-established fact that in counting limit cycles for planar, differential systems, the two components $P$ and $Q$ can be assumed not to have non-trivial common factors. 
It is also well-known that, under small perturbations in the coefficients
of $P$ and $Q$, the components of the algebraic curve $\Dv=0$ become
homeomorphic either to a straight line or to an oval, and that the
number of contact points, understood as solutions of system  \eqref{salto0}, of the vector field $(P,Q)$ with the curve
$\Dv=0$, is finite, and they all are simple. We refer to this case as the generic situation.
Given that the bound for the generic case is uniform on the degree $n$ of the system, showing that such an upper bound for the
number of limit cycles of a polynomial differential system of degree
$n$ extends to a non--generic vector field of the same degree $n$ is not surprising. Such perturbation argument will be described in the final section. As indicated, it enables to extend bounds on generic differential systems of a certain degree to general systems of the same degree. Because of this remark, we will always take for granted those generic assumptions, knowing that they extend without trouble to the general situation. 

Our principal, fundamental concern is to prove Theorem \ref{principall}. If one assumes that it has been shown, then
Theorem \ref{t5} in the generic case is easy to prove. 
This amounts to calculating the value of the two parameters $M$ and $N$ in this result in terms of the degree $n$ of the system. 

To do so, we recall two
classical theorems:

\begin{theorem} [Bezout Theorem, \cite{Fu}]\label{Be}
Let $R(x, y)$ and $S(x, y)$ be two polynomials with coefficients in
$\R$. If both polynomials do not share a non-trivial common factor,
then the algebraic system of equations
$$
R(x, y)=S(x, y)=0
$$
has at most degree$(R)$degree$(S)$ solutions.
\end{theorem}

\begin{theorem}[Harnack Theorem, \cite{Gu}]\label{Harnack}
The maximum number of connected components of an algebraic curve of
degree $k$ is
\begin{itemize}
\item[(a)] $1+(k-1)(k-2)/2$ if $k$ is even,

\item[(b)] $(k-1)(k-2)/2$ if $k$ is odd.
\end{itemize}
\end{theorem}

\begin{proof}[Proof of Theorem \ref{t5} based on 
Theorem \ref{principall}] 
We need to find an upper bound for the
number $N$ of the solutions that system (\ref{salto0}) may have when
$P$ and $Q$ are polynomials of at most degree $n$. By Bezout's
theorem we have that 
$N\le 2(n-1)^2$,
because in the generic case we
discard the possibility that the two equations of system
\eqref{salto0} may have a non--trivial common factor. Note that the degree of the first equation in system \eqref{salto0} is $2n-2$ while that of the second one is $n-1$.

By Theorem \ref{Harnack} the number $M$ of components of $\Dv=0$
satisfies  
$M\le(1/2)(n-2)(n-3)+1$,
if $n$ is even, and 
$M\le (1/2)(n-2)(n-3)$,
if $n$ is odd. Note that $k=n-1$. 

The final number in the statement of the theorem is then a direct
consequence of Theorem \ref{principall}, i.e. 
\[
\begin{array}{rl}
1+(n-1)^2(N+M) \le& 1+(n-1)^2\left(\dfrac12(n-2)(n-3)+1+2(n-1)^2\right)
\vspace{0.2cm}\\
=& \dfrac52 n^4-\dfrac{23}2 n^3+ \dfrac{43}2n^2-\dfrac{37}2n+7,
\end{array}
\]
if $n$ is even, while
\[
\begin{array}{rl}
1+(n-1)^2(N+M) \le& 1+(n-1)^2\left(\dfrac12(n-2)(n-3)+2(n-1)^2\right)
\vspace{0.2cm}\\
=& \dfrac52 n^4-\dfrac{23}2 n^3+ \dfrac{41}2n^2-\dfrac{33}2n+6,
\end{array}
\]
if $n$ is odd. This yields the numbers in the statement of Theorem
\ref{t5} in the generic case.
\end{proof}

Two corollaries of Theorem \ref{principall} are worth stating. The first refers to quadratic differential systems which have attracted a lot of attention throughout the years as a training ground for novel ideas and techniques (\cite{Bam}, \cite{duelrou}, \cite{durourou}, \cite{rou}, \cite{rouzhu}, \cite{Sh}, \cite{smits}).

\begin{corollary}
If the divergence of a quadratic polynomial differential system
\eqref{e1} is constant or zero, then it has no limit cycles.
Otherwise if the straight line $\Dv=0$ of a quadratic polynomial
differential system \eqref{e1} has:
\begin{itemize}
\item[(a)] either one or two contact points, then it cannot have more than  $4$
limit cycles.
\item[(b)] no contact points, then it has no limit cycles.
\end{itemize}
At any rate, $H(2)=4$.
\end{corollary}

\begin{proof}
The set $\Dv=0$ for a quadratic polynomial differential system is
either empty, a straight line, or the whole plane. If it is empty,
i.e. if $\Dv$ is a non-zero constant, then the system has no limit
cycles by Bendixon criterium (see for instance Theorem 7.10 of
\cite{DLA}). If it is the whole plane, the system is Hamiltonian and
so it has no limit cycles. We assume that $\Dv=0$ is a straight line.
So using \eqref{cotaprin}, we have 
$$
n=2,\quad M=1, \quad N\in\{0, 1, 2\}.
$$

\noindent If $N=2$, contact points are simple, and then 
$1+(n-1)^2(M+N)=1+1+2=4$.
If $N=1$, the unique contact point has multiplicity two, and the upper bound is the same. 
If $N=0$, then there are no contact points and the limit cycles
cannot intersect the straight line $\Dv=0$. Again by Bendixon
criterium, the system has no limit cycles.

Since quadratic systems with four limit cycles are known (\cite{CW}, \cite{Sh}), we conclude that indeed $H(2)=4$. 
\end{proof}

For Li\'enard systems (\cite{caudu}, \cite{limepu}), we have the following. 

\begin{corollary}
We consider the Li\'enard  polynomial differential systems
\begin{equation}\label{lienard}
\dot x= P(x, y)=y-f(x),\quad \dot y= Q(x, y)=g(x),
\end{equation}
where $p$ is the degree of $f$, and $q$ is the degree of $g$. So $n=
\max\{ p, q\}$. A system \eqref{lienard} cannot have more than 
$$
1+2  (\max\{p, q\}-1)^2(p-1)
$$ 
limit cycles.
\end{corollary}

\begin{proof}
It is well known that a system \eqref{lienard} has at most $p-1$
connected components because the curve  
$\Dv= f'(x)= 0$ 
corresponds to
the critical values of the polynomial $f$. Note that each component
is a vertical straight line in the $(x,y)$--plane.  System
\eqref{salto0} becomes
\[
(y-f(x))f''(x)=0, \qquad f'(x)=0,
\]
for differential system \eqref{lienard}. If 
$f''(x)=f'(x)=0$
has a solution $x_0$,
the vertical straight line $x=x_0$ is formed by contact points, so it
is invariant and we do not need to take it into account, because limit
cycles cannot intersect such straight line. Therefore a connected
component of the curve $\Dv=0$ has one single contact point
$(x_0,f(x_0))$ for each zero $x_0$ of the polynomial $f'(x)$ such
that $f''(x_0)\ne 0$. Using again \eqref{cotaprin}, we have 
$$
n= \max\{p, q\},\quad M=N\le p-1.
$$
Hence
\[
1+(n-1)^2 (M+N)  =1+2  (\max\{p, q\}-1)^2(p-1).
\]
This completes the proof.
\end{proof}

\subsection{Driving idea for the proof of main theorem}
Our guiding principle for the proof of Theorem \ref{principall}, expressed in a single sentence, is:
\begin{quote}
Limit cycles of the differential system \eqref{e1}, once parameterized in $[0, 1]$ for normalization, can be viewed as zeros of the following functional associated with it in a natural way
\begin{equation}\label{germen}
E_0(x, y)=\int_0^1 \frac12(P(x, y)y'-Q(x, y)x')^2\,dt.
\end{equation}
Use then Morse inequalities to count them in terms of its critical paths. 
\end{quote}

Even if the following parallelism may be blamed as too naive, it may help some readers stay oriented about where we are heading. It is exactly  a similar situation in dimension one. 
\begin{quote}
If we are interested in knowing how many real roots a single equation 
$f(x)=0$
may posses, one possibility is to look for global minimizers of the function 
$g(x)=(1/2)f(x)^2$.
It is obvious that roots of the first correspond exactly to absolute minimizers of the second, given that $g\ge0$ always. Critical points for $g$ are roots of 
$f(x)f'(x)=0$;
and those not corresponding to absolute minimizers 
$f(x)\neq0$,
are  zeros of $f'(x)$. Morse inequalities in this simple context reduce just to classical Rolle's theorem that yields the upper bound
$$
\#(\hbox{roots of $f$})\le 1+\#(\hbox{roots of $f'$}\setminus\hbox{roots of $f$}).
$$
This last inequality is the bottom line of our perspective. 
\end{quote}

Our contribution consists in an attempt to rigorously show that the simple idea just expressed can be carried out to its fulfillment for the proof of Theorem \ref{principall}. There are two fundamental parts to cover our program:
\begin{enumerate}
\item show that the counting procedure is valid; and
\item find the desired upper bound.
\end{enumerate}

\subsection{Final comments} Though very well-known to specialists, Hilbert's 16th problem \cite{Hi} may be not known to mathematicians working in other areas. As described above, we are interested in the second part of it which
deals with the number of limit cycles of planar, polynomial differential systems of the form
$$
x'=P(x, y),\quad y'=Q(x, y),
$$
where both $P$ and $Q$ are polynomials of two variables, and the possibility of providing an upper bound on them depending solely on the maximum degree $n$ of both $P$ and $Q$. 

The history of such a problem along the XX-th century is one of the most fascinating situations in Mathematics one can look at. We will not spend time on this here as there are very good and reliable accounts; particularly \cite{A1}, \cite{A2}, \cite{BNY}, \cite{CL}, \cite{I2}, \cite{Ka}, \cite{Li}, \cite{Ll}, \cite{Yak}, \cite{smale} are indispensable to learn and understand such history. It is also appropriate to recall that, according to Smale \cite{Sm}, except for the Riemann hypothesis, the second part of Hilbert's 16th problem seems to be the most elusive of Hilbert's problems. 
Aside from the various failed attempts to solve the problem, there are essentially two sources to examine this problem. Both \cite{Ec}, \cite{Il} claim the finiteness of the number of limit cycles for any planar, polynomial differential system without establishing a bound in terms of the degree of the system. Again according to Smale (\cite{Sm}), ``these two papers have yet to
be thoroughly digested by the mathematical community". 

Our contribution is written in a self-contained form for the most part of it. Proofs, though, of very classic results in Analysis which are part of the usual background for many non-linear analysts are not covered as they will, most likely, be accepted without discussion when invoked. 

Particular mention deserves reference \cite{openai}. Among the hundreds of mathematical problems claimed to have been solved, the second part of Hilbert's 16th problem can be found there too. Yet no explicit upper bound is provided. It is not shown if such a bound could be polynomial in the degree of the system or what the degree of such polynomial might be.

\section{Strategy and organization}
Our starting point is functional $E$  in \eqref{germen}, namely
$$
E(\bu)=\frac12\int_0^1(P(x(t), y(t))y'(t)-Q(x(t), y(t))x'(t))^2\,dt,\quad \bu=(x, y),
$$
defined in the natural Hilbert space $H^1([0, 1]; \R^2)$. 
It is elementary to realize that $E(\bu)=0$ for any parametrization $\bu$ of a limit cycle of differential system \eqref{e1} in the interval $[0, 1]$. Our driving idea is to use Morse inequalities to bound the number of absolute minimizers of $E$ by the number of its critical points. 
There is a standard way to do so. It is like a big, global Rolle's theorem in infinite dimension as already indicated earlier. 
\begin{theorem}[Morse inequalities]\label{morse}
Let $E:\L\to\R$
be a $\mathcal{C}^2$-functional defined over a Hilbert space $\L$, which is bounded from below, coercive, enjoying the Palais-Smale property, and having a finite number of critical points, all of which are non-degenerate and of a finite index. Put $M_k$ for the (finite) number of critical points of $E$, for each fixed index
$k$.  Then
\begin{gather}
M_0\ge1,\quad M_1-M_0\ge-1,\quad M_2-M_1+M_0\ge 1,\quad \dots,\nonumber\\
\sum_{k=0}^\infty (-1)^kM_k=1.\label{morsein}
\end{gather}
\end{theorem}
There is a number of fundamental concepts in this statement that need to be examined before its conclusion is utilized. To avoid breaking the thread of our discussion at this stage, let us ignore them for the time being and see how one could use this important result for $E$. 

Those zeroes of $E$ represented by limit cycles are part of the class of absolute minimizers of $E$ (for there might be other possibilities for zeroes of $E$), which in turn is a subclass of the full set of local minimizers. The number of local minimizers is precisely $M_0$. Therefore, from \eqref{morsein}, we find
$$
M_{abs}+(M_0-M_{abs})=M_0=1+\sum_{k=1}^\infty (-1)^{k+1}M_k\le 1+\sum_{k=1}^\infty M_k,
$$
if $M_{abs}(\le M_0)$ is the number of absolute minimizers. Finally
$$
M_{abs}\le 1+\sum_{k=1}^\infty M_k+(M_0-M_{abs}).
$$
The sum on the right-hand side is precisely the number of all critical points of $E$ which are not global minimizers. If we let $M_{cri}$ be the number of those critical points, we have the upper bound
\begin{equation}\label{cotasup}
M_{abs}\le 1+M_{cri}.
\end{equation}
If we are able to find an upper bound for $M_{cri}$ in terms of the degree $n$ of the system, we would have an upper bound for $M_{abs}$, i.e. an upper bound on the number of zeroes of $E$ and, consequently, on the number of limit cycles of our polynomial system. 

\subsection{Difficulties} When trying to use Theorem \ref{morse} for $E$ in the initial Hilbert space of planar curves $\espaciou$, one encounters many troubles from the outset. To begin with, even if $E$ is bounded from below, it is not coercive in $\espaciou$. But a much more severe difficulty is concerned with the fact that the image of curves in $\espaciou$, understood as subsets of $\R^2$, may admit infinitely-many re-parameterizations even if we discard the possibility of running images several times around counter- or clock-wise. This implies that there is no hope to apply Theorem \ref{morse} to our $E$ in $\espaciou$, for all indexes $M_k$ occurring in its statement are infinite in that space. We have not even mentioned the fundamental Palais-Smale property that must be ensured as well. 

The canonical alternative facing such an apparent dead-end consists of three new main ingredients:
\begin{enumerate}
\item Enlarge the analytical scenario by the consideration of another finer Hilbert space $\H\subset\espaciou$, and a suitable perturbation $E_\e:\H\to\R^+$, $\e>0$, of $E_0=E$ in an attempt to comply with the appropriate properties for functionals in Theorem \ref{morse}.
\item Restrict attention to an increasing family of open subsets $\bbO_d$ in $\espacioo$, each of which is invariant under the flow of $-E_\e$, where intricate, undesirable behavior is discarded. 
\item Interpret limit cycles of our polynomial, differential system as determining different connected components of $\{E_0=0\}$, and show that this one-to-one labelling by limit cycles is maintained when passing from components of $\{E_0=0\}$ to components of $\{E_\e\le a\}\cap\bbO_d$ for $a$, $\e$ and $d$ suitably chosen. 
\end{enumerate}

The introduction of the perturbation $E_\e$ on $\H$ changes the landscape of critical points with respect to those of $E_0$. It also disfigures components of the set $\{E_0=0\}$ since the zero set for $E_\e$ is, in general, empty; and so we need to focus rather on sub-level sets of the form $\{E_0\le a\}$ for $a$, positive and small, and understand the relationship between components of $\{E_0=0\}$ and $\{E_\e\le a\}$. The important issue for our counting method is to ensure that each component of the zero set of $E_0$ represented by a limit cycle of our differential system is still identified in a one-to-one manner with components of $\{E_\e\le a\}$ in $\bbO_d$. 
This is a crucial point in our approach since it is not true, in general if we work in all of $\H$, that components of $\{E_0=0\}$ identify in a one-to-one way components of $\{E_\e\le a\}$ for positive $a$ and $\e$.

\subsection{Our guiding principle}\label{centralz} 
The solution to the indicated difficulties starts with a version of Theorem \ref{morse} valid on invariant subsets. See subections \ref{princon} and \ref{morin}.

\begin{theorem}[Morse inequalities for invariant sets]\label{morsevalle}
Let $E:\H\to\R$ 
be a $\mathcal{C}^2$-functional defined over a Hilbert space $\H$, which is bounded from below, coercive, enjoying the Palais-Smale property, and having a finite number of critical points, all of which are non-degenerate and of a finite index. Let $\bbI\subset\H$ be open, topologically equivalent to a ball, and invariant under $E$. Put $M_k(\bbI)$ for the (finite) number of critical points of $E$ in $\bbI$, for each fixed index
$k$.  Then
\begin{gather}
M_0(\bbI)\ge1,\quad M_1(\bbI)-M_0(\bbI)\ge-1,\quad M_2(\bbI)-M_1(\bbI)+M_0(\bbI)\ge 1,\quad \dots,\nonumber\\
\sum_{k=0}^\infty (-1)^kM_k(\bbI)=1.\nonumber
\end{gather}
\end{theorem}
This statement ensures that our previous bound \eqref{cotasup} is still valid 
\begin{equation}\label{morseinvalle0}
M_{abs}(\bbI)\le 1+M_{cri}(\bbI)
\end{equation}
when attention is restricted to an invariant set $\bbI\subset\H$ for the functional $E$.  

For the sake of notation, we adopt the following standard definition. For the time being, we will take all properties in the statement blindly. We will clearly discuss them later at the right time, and with the necessary care. 

\begin{definition} \label{mf}
A functional $E:\H\to\R$
is called a Morse functional if it is
$\C^2$-, non-negative, coercive, enjoys the Palais-Smale property,
and has a finite number of critical points over each sub-level set
$\{E\le c\}$ for each non-critical value $c$, all of which are
non-degenerate and with a finite index.
\end{definition}

The following result is our guiding principle to legitimate our counting procedure. Once our setting enables to apply it, our task will be to furnish the sought upper bound. 

There are several interconnected parameters in the statement which may make it a bit confusing. These are: $a$, which is a certain small, positive value; $\e$, which is doomed to be quite small, depending on $a$; and a large positive integer $d\in\N$ depending on $\e$. 
\begin{theorem}\label{morsefinal}
Let $E_\e:\bbH\to\R^+$, $\e\ge0$, 
be a family of smooth, non-negative, coercive functionals, with $0\le E_0\le E_\e$. Let $\bbO_d$,  $d\in\N$, be an increasing family of  open subsets of $\bbH$ which are topologically equivalent to a ball, and put 
$$
\bbO=\bigcup_{d\in\N}\bbO_d, \quad \bbO_{\e, d}=\bbO_d\cap\{E_\e\le b_\e\},
$$
for some sequence of values $b_\e\to+\infty$ as $\e\to0^+$. Let $\bbP\subset\bbO$ be such that different elements of $\bbP$ are contained in different components of $\{E_0=0\}$ in $\bbO$.
Suppose, in addition, that:
\begin{enumerate}
\item $E_\e$ is a Morse functional for positive $\e$;
\item $\bbO_{\e, d}$ is invariant for $E_\e$ for all $\e$, and for large $d$ (possibly depending on $\e$); and
\item for $a$ sufficiently small, $\e$ small enough (depending on $a$), and all large $d$ (depending on $\e$), 
elements of $\bbP$ in $\bbO_{\e, d}$ identify (are contained in), in a one-to-one manner, components of $\{E_\e\le a\}\cap\bbO_{\e, d}$, and each of these components is topologically equivalent to a ball.
\end{enumerate}
Then 
\begin{equation}\label{desimportante}
\#(\bbP)\le 1+\lim_{a\to0}\lim_{\e\to0}M_{cri, \e, a},
\end{equation}
if $M_{cri, \e, a}$ is the number of critical elements of $E_\e$ in $\{E_\e\ge a\}\cap\bbO$.
\end{theorem}
\begin{proof}
Pick up an arbitrary, finite subset $\bbP'$ of $\bbP$, so that for $d$ large enough and sufficiently small $\e$, $\bbP'\subset\bbO_{\e, d}$. 
By hypothesis, for $a$ sufficiently small and non-critical for $E_\e$\footnote{By Sard's theorem, critical values of smooth functionals have measure zero. We will recall and utilize this fact later.}, $\e$ smaller if necessary (depending on $a$), and all large $d$ (depending on $\e$), components of $\{E_\e\le a\}\cap\bbO_{\e, d}$ enjoy the properties in the statement.

Let $\bbC_{i, a, \e, d}$, $i\in I$, be the full set of disjoint components of $\{E_\e\le a\}$ in $\bbO_{\e, d}$, some of which will correspond to those associated with the selected elements in $\bbP'$. 
We can therefore apply Theorem \ref{morsevalle} to $E_\e$ and each $\bbC_{i, a, \e, d}$, separately for each $i\in I$, as well as in all of $\bbO_{\e, d}$,
to find that
\begin{equation}\label{igualdades}
\sum_k (-1)^k M_k(\bbC_{i, a, \e, d})=1,\quad \sum_k (-1)^k M_k(\bbO_{\e, d})=1,
\end{equation}
for each $i$. The sum in \eqref{morseinvalle0} is additive with respect to disjoint sets where it is considered, and hence
$$
\sum_i\sum_k(-1)^kM_k(\bbC_{i, a, \e, d})+\sum_k(-1)^kM_k(\bbO_{\e, d}\setminus\cup_i\bbC_{i, a, \e, d})=\sum_k (-1)^k M_k(\bbO_{\e, d}).
$$
We conclude, because of \eqref{igualdades} and the fact that each element of $\bbP'$ corresponds to a unique $i\in I$, that
$$
\#(I)+\sum_k(-1)^kM_k(\bbO_{\e, d}\setminus\cup_i\bbC_{i, a, \e, d})=1,
$$
and
$$
\#(\bbP')\le\#(I)\le 1+\sum_kM_k(\bbO_{\e, d}\setminus\cup_i\bbC_{i, a, \e, d}).
$$
The sum that remains is less than or equal to the total number of critical points of $E_\e$ off the union of all components of $\{E_\epsilon\le a\}\cap\bbO{\e, _d}$, which is certainly smaller than those in the set
$\{a\le E_\e\}\cap\bbO$.
If we put $M_{cri, \e, a}$ for such total number of critical points of $E_\e$ in $\bbO$ with critical value uniformly away from zero, i.e. the number of critical points of $E_\e$ in the subset $\{a\le E_\e\}\cap\bbO$, then we recover the upper bound
\begin{equation}\label{cotasupe}
\#(\bbP')\le 1+M_{cri, \e, a}.
\end{equation}
The arbitrariness of $\bbP'$, and the ordered smallness of $a$ and $\e$ finish the proof.
\end{proof}

If the number $M_{cri, \e, a}$ admits further an upper bound $M_{cri}$ independent of $\e$ and of $a$, as they become small, respectively, in the appropriate order, in terms of the degree $n$, then the full set $\bbP$ is finite and
$$
\#(\bbP)\le 1+M_{cri}.
$$
This uniform upper bound is another crucial step of our approach. 

\subsection{Our particular framework}
The natural Hilbert space for our functional $E_0$ is
$\L=H^1_O([0, 1]; \R^2)$, the space
of periodic paths, $\bu(0)=\bu(1)$ (but this common value is undetermined), with a weak first derivative $\bu'(t)$ which is square-integrable in the unit interval
$$
\int_0^1|\bu'(t)|^2\,dt<\infty.
$$
The integrand corresponding to \eqref{germen} is
\begin{equation}\label{germencon}
F(t, \bu, \bv)=\frac12(\bv\cdot\bF^\perp(\bu))^2,\quad \bF=(P, Q), \bF^\perp=(-Q, P),
\end{equation}
so that
$$
E_0(\bu)=\int_0^1 F(t, \bu(t), \bu'(t))\,dt.
$$

As discussed before, one needs to perturb $E_0$ and pass down to a better space of paths, namely
\begin{equation}\label{perturbacion1}
\H=\espacioo,\quad E_\e(\bu)=E_0(\bu)+\frac{\e}2\|\bu-\bv_0\|^2_{\H}.
\end{equation}
$\espacioo$ is the space of paths with a weak second derivative $\bu''(t)$ that is square-integrable
$$
\int_0^1|\bu''(t)|^2\,dt<\infty,
$$
which are periodic
$$
\bu(0)=\bu(1),\quad \bu'(0)=\bu'(1).
$$
The specific path $\bv_0\in\H$ is to be chosen in an appropriate way to ensure that Theorem \ref{morsevalle} can, this time, be applied to $E_\e$ in $\H$. In this way, the perturbed functional becomes of second order explicitly given by
$$
E_\epsilon(\bu)=E_0(\bu)+\frac\e2\int_0^1[|\bu''(t)-\bv''_0(t)|^2+|\bu'(t)-\bv'_0(t)|^2+|\bu(t)-\bv_0(t)|^2]\,dt.
$$
In addition to this passage to second-order problems, there are three main issues worth realizing.
\begin{enumerate}
\item The choice for $\bbO_d$. Recall that $\bbO$ is the non-decreasing union of the $\bbO_d$'s. 
Since our main concern is focused on limit cycles of planar, polynomial vector fields, we will consider the set $\bbO_d$ as the subset of $\espacioo$ of non-singular paths (or parameterizations) with winding number $+1$, and a maximum number $d$ of full rounds of the tangent vector in either sense. The concept of winding number goes back to Whitney \cite{whitney} who formalized it and proved its main properties. The limit cycles we are interested in counting are represented by parameterizations with winding number +1 (run counter-clockwise). Our counting procedure will take place in $\bbO_d$ for some finite $d$. 
\item Another main point in this part of our development is to understand how to set up and manipulate the equation for critical paths of integral functionals as the ones written above. It is something well-known that we are talking about systems of ODEs of order four. 
\item The fact that the perturbed functionals $E_\e$ are of the form \eqref{perturbacion1} with a norm in $\H$ which is finer than the norm in $\L$, carries us to face singularly-perturbed problems and their asymptotic behavior as $\e\to0$. This is a delicate area that asks for fine points in arguments. 
\end{enumerate}

Our efforts, then, are directed towards selecting the varios ingredients in a way that the hypotheses of Theorem \ref{morsefinal} can be shown to hold, and thus validate our counting procedure. Secondly, we will find an upper bound of critical points independent of the parameters involved. The outcome will be the complete proof of Theorem \ref{principall}. 

In Section \ref{tres}, we will deal with some abstract results. In addition to Morse inequalities, our statements pretend to ensure that a good choice of the auxiliary path $\bv_0$ is possible. Some other specific results for singularly-perturbed problems in this section will later be invoked. Section \ref{cuatroz} focuses on general facts that will be specifically utilized in our particular setting. These include the main result about the equation of critical paths for $E_\e$ and some other important tools to deal with them, as well as some elementary lemmas to treat the possibility of multiple solutions. Section \ref{concretos} focuses on proving some important properties of our specific functional $E_0$ so that results in Section \ref{tres} can be applied. Finally, the two main steps for the uniform upper bound are treated in Sections \ref{seisz} and \ref{contar}. Again the analysis in these sections is based on technical results that have been postponed until Section \ref{cuatroz}.

\section{Fulfillment of hypotheses in our framework}

\subsection{Some analytical preliminaries}\label{prelim}
We briefly state here in a quite explicit way some basic notions about spaces of functions with weak derivatives having suitable integrability properties, as well as recalling again concepts like the coercivity of a functional. It may be convenient to do so for the sake of some interested readers who may not be familiar with these concepts. 
We refer to \cite{brezis} for a main, accesible source in this regard, and much more related information. 

Aside from end-point conditions, the underlying natural Hilbert space for $E_\e$ is
$$
\espacio=\left\{(x, y):[0, 1]\to\R^2: \int_0^1[x^2+y^2+(x')^2+(y')^2+(x'')^2+(y'')^2]\,dt<\infty\right\}.
$$
This is nothing but the classical Sobolev space 
of paths with square-integrable weak derivatives up to order two. The
inner product in this space is
$$
\langle(x_1, y_1), (x_2,
y_2)\rangle=\int_0^1(x_1x_2+y_1y_2+x'_1x'_2+y'_1y'_2+x''_1x''_2+y''_1y''_2)\,dt,
$$
and the associated norm
$$
\|(x, y)\|^2=\int_0^1[x^2+y^2+(x')^2+(y')^2+(x'')^2+(y'')^2]\,dt.
$$
Norms and inner products occurring henceforth are meant to be these.
Paths in $\espacio$ have continuous first derivatives. Since parameterizations of limit cycles as integral curves of the corresponding polynomial differential system, suitably normalized to the unit interval, are $\C^\infty$, they belong to this space.

We include the following basic convergence properties (\cite{brezis}). 
\begin{proposition}\label{convergenciauniforme}
Paths in $H^1([0, 1]; \R^2)$ are absolutely continuous. Paths in $H^2([0, 1]; \R^2)$ are $\cC^1$ with a derivative $\bu'$ which is absolutely continuous. If $\bu_j\to\bu$ in $\espacio$, then 
$$
\bu_j\to\bu, \quad \bu'_j\to\bu'
$$
uniformly in $[0, 1]$. 
\end{proposition}
\begin{proof}
Note that
$$
|\bu(r)-\bu(s)|\le\int_s^r|\bu'(t)|\,dt\le\sqrt{r-s}\left(\int_0^1|\bu'(t)|^2\,dt\right)^{1/2}
$$
for $\bu\in H^1([0, 1]; \R^2)$. Hence $\bu$ is absolutely continuous. The same inequality replacing $\bu$ and $\bu'$ by $\bu'$ and $\bu''$, respectively, when $\bu\in H^2([0, 1]; \R^2)$, shows the second part.
\end{proof}

We recall that coercivity for a general functional $E$ defined in a Hilbert space $\H$ means that
$$
E(\bu)\to+\infty\quad\hbox{ as }\quad\|\bu\|\to\infty\hbox{ with }\bu\in\H.
$$
If a functional $E_0$ defined in $\espacio$ is non-negative, the perturbation
$$
E_\e(\bu)=E_0(\bu)+\frac{\e}2\|\bu-\bv_0\|^2
$$
automatically becomes coercive for every positive $\e$, and every fixed element $\bv_0$. 

To summarize in a compact form our situation, we will concentrate on the family of functionals
$$
E_\e:\espacioo\to\R^+
$$
where
\begin{gather}
\espacioo=\{\bu\in\espacio: \bu(0)=\bu(1), \bu'(0)=\bu'(1)\},\nonumber\\
\langle\bu, \bv\rangle=\int_0^1(\bu(t)\cdot\bv(t)+\bu'(t)\cdot
\bv'(t)+\bu''(t)\cdot\bv''(t))\,dt,\nonumber\\
\|\bu\|^2=\|\bu\|^2_{H^2([0, 1]; \R^2)}=
\int_0^1\big(|\bu''(t)|^2+|\bu'(t)|^2+ |\bu(t)|^2\big)\,dt,\nonumber\\
E_0(\bu)=\frac12\int_0^1 (\bF^\perp(\bu)\cdot\bu')^2\,dt,\nonumber\\
E_\e(\bu)=E_0(\bu)+\frac\e2\|\bu\|^2-\langle\bu,
\bv_0\rangle+\frac\e2\|\bv_0\|^2=E_0(\bu)+\frac\e2\|\bu-\bv_0\|^2,\nonumber
\end{gather}
and
\begin{equation*}\label{primsisdin}
\begin{array}{c}
\bu=(x, y), \quad \bF(\bu)=(P(x, y), Q(x, y)),\\
\bF^\perp(\bu)= (-Q(\bu),P(\bu)).
\end{array}
\end{equation*}

Once again, a fundamental fact for us to bear in mind is that paths in $\espacioo$ are $\C^1$, and convergence in $\espacio$ implies uniform convergence of first derivatives.

Theorem \ref{morsefinal} is our guiding principle. Our main concern is to meet all the requirements in that statement to validate our counting procedure. 

\subsection{A family of Morse functionals}
The first fundamental issue is to show that each $E_\e$ is a Morse functional when the auxiliary path $\bv_0$ is suitably chosen (Definition \ref{mf}). According to the abstract results in Section \ref{tres}, all we have to check is the compactness of the derivatives of the base functional $E_0$. Since these are purely technical issues, we have deferred them until the final section. 

\begin{corollary}
A path $\bv_0\in \espacioo$ can be appropriately chosen, such that functionals $E_\e$ above are Morse functionals for all positive $\e$. 
\end{corollary}
The proof of this corollary is a consequence of our later discussions in subsections \ref{pertur},  \ref{mascon}, \ref{prueprin}, as well as Section \ref{concretos}. 
\subsection{Ambient set for our analysis}\label{vallee}
Another essential ingredient of our analysis focuses on isolating a suitable increasing family of subsets 
$$
\bbO_d\subset\espacioo, \quad d\in\N,
$$
where images of paths cannot admit different reparameterizations covering such images more than once either counter- or clockwise. 
To make this idea rigorous, we will rely on important concepts and results in \cite{whitney}. Though some of those are rather classical, for the sake of readers we state them here with some care. 

\begin{definition}\label{whit}
A planar, parametrized regular closed curve is a continuously differentiable mapping 
$\bu(t):[0, 1]\to\R^2$
such that
$$
\bu(0)=\bu(1), \bu'(0)=\bu'(1),\quad |\bu'(t)|>0\hbox{ for all }t\in[0, 1].
$$
\end{definition}
Note how every parametrized regular curve belongs to $\espacioo$ according to Proposition \ref{convergenciauniforme}. 

For such a regular curve $\bu$, we can consider the normalized tangent vector
$$
\bn(t)=\frac1{|\bu'(t)|}\bu'(t):[0, 1]\to\bbS.
$$
\begin{definition}
\begin{enumerate}
\item The winding number of a regular curve $\bu$ is the total number of full turns, taking into account its sense, of the normalized tangent vector $\bn$ in the unit circle $\bbS$ as $t$ runs in the unit interval. 
\item The absolute winding number is the number of full turns, regardless of whether they are clock- or counterclockwise, that $\bn$ runs around $\bbS$ as $t$ runs through $[0, 1]$. 
\end{enumerate}
\end{definition}
In this way, a regular curve $\bu$ could have winding number $+1$, and yet its absolute winding number could be larger or even much larger. The point is that the tangent vector $\bn(t)$ turns fully around as many times clockwise as counter-clockwise plus one. 

The main result from \cite{whitney} shows that two regular curves can be continuously deformed into each other if and only if they share the winding number. 

\begin{definition}
We put $\bbO_d$ for the subset of $\espacioo$ of regular curves with a nowhere-null tangent vector, winding number $+1$, and absolute winding number not greater than $d$. 
\end{definition}

\begin{proposition}\label{bbO}
$\bbO_d$ is an open, convex subset of $\espacioo$. It is, therefore, topologically equivalent to a ball. 
\end{proposition}
\begin{proof}
Note that convergence in $\espacioo$ implies uniform convergence of tangent vectors (Proposition \ref{convergenciauniforme}). This is standard. The convexity is also clear. 
\end{proof}

\subsubsection{Invariance of $\bbO_d$ under $E_\e$}
This property means that the flow of $-E_\e$, leading to smaller values of $E_\e$, cannot push paths in $\bbO_d$ to a larger value of $d$. This looks intuitively pretty clear as such a behavior could only be favored by a much more erratic functional than $E_0$. 
Formally, we need to argue that the value of $E_\e$ at a path 
$\bu(t)\in\partial\bbO_d$ 
decreases when we perturb $\bu(t)$ locally in such a way that the resulting perturbation $\tilde\bu(t)$ lies back in $\bbO_d$. 

We will be using again (Section \ref{concretos}) the compactness of the derivative mapping
$$
E'_0(\bu): \espacioo\to\espacioo,
$$
which means that the sequence $\{E'_0(\bu_j)\}$ always admits a subsequence converging strongly in $\espacioo$ whenever $\{\bu_j\}$ is uniformly bounded in the same space. 

\begin{proposition}\label{invarianza}
Suppose the auxiliary path $\bv_0$ is smooth, and belongs to $\bbO_1$ ($d=1$). Let $\e>0$ be given. There is a bounded subset $\B_\e$ such that the intersection $\bbO_d\cap\B_\e$ is invariant under $E_\e$ for every sufficiently large $d\in\N$ (possibly depending on $\e$).
\end{proposition}
\begin{proof}
The parameter $\e$ is regarded as fixed. Since functional $E_\e$ is coercive, sub-level sets $\{E_\e\le b_\e\}$ for some large enough $b_\e$ are invariant (just as it was indicated in the proof of Theorem \ref{morsefinal}). Fix one such $b_\e$ in such a way that all critical paths of the Morse functional $E_\e$, a finite number of them, belong to such sub-level set. Put $\B_\e$ to mean $\{E_\e\le b_\e\}$. 

Suppose, seeking a contradiction, that we could find a subsequence of increasing values of $d\in\N$ with paths $\bu_d\in\partial\bbO_d\cap\B_\e$   such that $-E'_\epsilon(\bu_d)$ pushes $\bu_d$ off $\bbO_d$, for every such $d$. 
By the claimed compactness of the derivative operator $E'_0$ (Lemma \ref{compacidad}), for some subsequence if necessary, $\{E'_0(\bu_d)\}$ converges (strongly) in $\espacioo$ to some fixed path $\bU$ which belongs to $\bbO_r$ for some $r$. But then the combination
$\bu_d-\delta E'_0(\bu_d)$
is a small perturbation of a translation of $\bu_d$ by a small, fixed path, $\delta\bU\in\bbO_r$ and, hence, it cannot push $\bu_d$ off $\bbO_d$ for arbitrary small values of $\delta$, and large values of $d$. This is the contradiction. 

The effect of adding the perturbation 
$(\e/2)\|\bu-\bv_0\|^2$
to $E_0$ only enforces the behavior just described if $\bv_0$ is taken in $\bbO_1$, being this term a multiple of the (square of the) distance to an element of $\bbO_1$. 
This shows that $\bbO_d$ is in fact invariant for $-E_\e$ for all large $d$. 

The same argument also justifies that there cannot be a critical path of $E_\e$ on $\partial\bbO_d$ since such possibility would force $-E'_0$ to point off $\bbO_d$ which has been discarded above. Indeed, if
\begin{equation}\label{criticalz}
E'_0(\bu)+\e(\bu-\bv_0)=\bcero,\bu\in\partial\bbO_d,
\end{equation}
then the element
$$
\e(\bu-\bv_0)=-E'_0(\bu)
$$
should point off $\bbO_d$ because $\bu-\bv_0$ does. But being equal to $-E'_0(\bu)$,  this behavior has been discarded above, for large $d$. Hence \eqref{criticalz} is impossible for large $d$. 
\end{proof}

\subsubsection{Identification of components of $\{E_0=0\}$}
Recall that 
$$
E_0(\bu)=\int_0^1 \frac12Z^2\,dt,\quad Z=P(x, y)y'-Q(x, y)x',\quad \bu=(x, y).
$$
We take $\bbP\subset\bbO$ as the set of all limit cycles of our original, polynomial, differential system through some respective parameterizations in $\bbO$.

\begin{lemma}\label{componenteperturbado}
Let $\bbP'$ be a finite subset of $\bbP$. There is $a(\bbP')>0$ and $d(\bbP')\in\N$, such that each component of $\{E_0\le a\}\cap\bbO_d$, for $a<a(\bbP')$ and $d\ge d(\bbP')$, contains at most a unique element of $\bbP'$.
\end{lemma}
\begin{proof}
Given that $\bbP'$ is finite, we can take $d(\bbP')$ large enough so that $\bbP'\subset\bbO_d$ for $d\ge d(\bbP')$. Suppose, seeking a contradiction, that we could find two distinct elements of $\bbP'$ in the same component of $\{E_0\le \delta\}\cap\bbO_d$ for every positive $\delta$, and some $d$ large and fixed. That would imply that for every such $\delta$, there is a continuous path 
$\sigma_\delta(s): [0, 1]\to\bbO_d$,
joining two different components of $\{E_0=0\}$ in $\bbP'$ whose image is entirely contained in $\{E_0\le\delta\}\cap\bbO_d$. 

Let $\bK$ be a compact set separating in the plane two different limit cycles in $\bbP'$. By continuity, $\sigma_\delta$, belonging to $\{E_0\le\delta\}\cap\bbO_d$ and joining the two elements of $\bbP'$, must intersect $\bK$ at some point belonging to the image of a certain path $\bu_\delta(=\sigma_\delta(s_\delta), s_\delta\in[0, 1])$ with $E_0(\bu_\delta)\le\delta$. By the compactness of $\bK$, there must be an accumulation point $\bP\in \bK$ as $\delta\to0$, which in this way does not belong to any of the limit cycles in $\bbP'$. 
We thus have that 
$$
E_0(\bu_\delta)\le\delta,\quad \bu_\delta(0)\in \bK,\quad \bu_\delta(0)\to \bP\hbox{ as }\delta\to0,\quad \bP\in \bK,\quad \bu_\delta\in\bbO_d.
$$
Put
$$
Z_\delta(t)=P(\bu_\delta)y'_\delta-Q(\bu_\delta)x'_\delta,\quad \bu_\delta=(x_\delta, y_\delta),
$$
a concrete function of $t$, and notice that $Z_\delta\to0$ in $L^2(0, 1)$, as $\delta\to0$, precisely because $E_0(\bu_\delta)\to0$. The path $\bu_\delta$ turns out to be the unique solution of the problem
$$
P(X, Y)Y'-Q(X, Y)X'=Z_\delta(t)\hbox{ in }[0, 1],\quad (X(0), Y(0))=\bu_\delta(0).
$$
By the indicated convergence of the right-hand side $Z_\delta$, the fact that $\bu_\delta\in\bbO_d$ for all $\delta$ and some fixed $d$, and the convergence of $\bu_\delta(0)$, $\bu_\delta$ must converge to the unique solution $\bu=(x, y)$ of
$$
P(x, y)y'-Q(x, y)x'=0\hbox{ in }[0, 1],\quad (x(0), y(0))=\bP\in\bK.
$$
This is not possible since we would have $E_0(\bu)=0$ and $\bu(0)=\bP$, and there is no limit cycle starting at the set $\bK$. 
This is the sought contradiction, and ends the proof. 
\end{proof}

Note that if we would allow the possibility that $d$ might grow indefinitely as $\delta\to0$, the claimed convergence above would be wrong in general. This is, in fact, the main reason why one is forced to work on subsets $\bbO_d$ for fixed (though possibly large) $d$ to ensure the appropriate compactness guaranteeing the convergence $\bu_\delta\to\bu$.

If we now realize that  
\begin{equation}\label{inclusion}
\{E_0=0\}\cap\bbO_d\subset \{E_\e\le a\}\cap\bbO_d\subset \{E_0\le a\}\cap\bbO_d,
\end{equation}
for arbitrary $a>0$, $d\in\N$, and $\e$ sufficiently small (depending on $a$), we see that for $a$ sufficiently small and every large $d$, once we focus on an arbitrary finite subset $\bbP'$, components of the set on the left-hand side identify, in a one-to-one manner, components of the intermediate level set according to the above lemma. 

\subsubsection{Components of $\{E_0=0\}\cap\bbO_d$ are topologically equivalent to a ball}
Recall that this set includes all $1$-periodic, $\cC^1$-, non-singular parameterizations of paths made up of pieces of integral curves with absolute winding number not greater than $d$. 

Given any $1$-periodic path 
$\bu(t)\in\bbO_d\cap\{E_0=0\}$,
the connected component determined by it corresponds to the various reparametrizations of $\bu(t)$ of the form  $\bu(\sigma(s))$ for a suitable class of functions $\sigma(s)$ to maintain it within $\bbO_d$. More specifically, we consider 
$$
\Lambda=\left\{\sigma(s): [0, 1]\to\R: \sigma' \hbox{ is absolutely continuous},
\sigma'>0, \int_0^1\sigma'(s)\,ds=1\right\}.
$$
Note that $\Lambda$ is a convex set. Recall that curves in $\bbO$ are $\cC^1$ and cannot have a vanishing tangent vector, hence they can never turn around.

\begin{lemma}\label{componentes}
Let $E_0(\bu)=0$, for $\bu\in\espacioo\cap\bbO_d$. 
\begin{enumerate}
\item The component of $\{E_0=0\}$ in $\bbO_d$ determined by $\bu$ is
$$
\Lambda_\bu\equiv\{\bu(\sigma(s)): \sigma\in\Lambda\}.
$$
\item The set $\Lambda_\bu$ is topologically equivalent to a ball.
\end{enumerate}
\end{lemma}
\begin{proof}
The first assertion is clear because winding and absolute winding numbers are not related to parameter dependence, but only to taken-on values of the normalized tangent vectors.  
For the second part, note that because $\Lambda$ is a convex set, it is topologically equivalent to a ball. $\Lambda_\bu$, being the image under a continuous mapping of $\Lambda$, must also be topologically equivalent to a ball.
\end{proof}

\subsubsection{Components of $\{E_\e\le a\}\cap\bbO_d$ are topologically equivalent to a ball}
Our first observation is that the proof of Lemma \ref{componenteperturbado} implies that the connected component of $\{E_0\le a\}$ containing a certain limit cycle cannot be arbitrarily close to the component of a different limit cycle for $a$ arbitrarily small. This statement, together with Lemma \ref{componentes}, lets us assume, without loss of generality, that for $a$ arbitrarily small, the connected components of $\{E_0\le a\}$  associated with limit cycles can be regarded as topologically equivalent to a ball as they cannot be one inside the other, because that is imposible in the limit as $a\searrow0$ by Lemma \ref{componentes}. 

Finally, by means of Corollary 6.2 in \cite{chang}, we conclude that suitable values for $a$, arbitrarily small, can be chosen so that $\{E_\e\le a\}\cap\bbO_d$ are strong retracts of $\{E_0\le a\}\cap\bbO_d$, and so they are topologically equivalent. Note that it suffices for our purposes to restrict attention to a sequence of values for $\e$ converging to zero, and that all results in Morse Theorey are valid on invariant sets of functionals. 

\subsection{Conclusion}
As a consequence of our work in this section, and bearing in mind the abstract results of Section \ref{tres} oriented towards the choice of the auxiliary path $\bv_0$, we are entitled to apply Theorem \ref{morsefinal} to our particular situation and conclude the following upper bound for the number of limit cycles of our original polynomial differential system
\begin{equation}\label{cotapreliminar}
\#(\bbP)\le 1+\lim_{a\to0}\lim_{\e\to0}M_{cri, \e, a}.
\end{equation}

\section{The uniform upper bound}\label{seisz}
We remind readers that in this central section, we will be invoking results in subsection \ref{percri}, and Section \ref{cuatroz}. 

\subsection{The equation for  critical closed paths}\label{aqui}
The application of Theorem \ref{basicsecond}, which provides the form of the critical path equation for second-order integral functionals, to our situation where
\begin{align}
F(t, \bu, \bz, \bZ)=&\frac12(\bF^\perp(\bu)\cdot\bz)^2+\frac\e2(|\bZ|^2+|\bz|^2+|\bu|^2)\nonumber\\
&+\bu\cdot\bv_0(t)+\bz\cdot\bv'_0(t)+\bZ\cdot\bv''_0(t)\nonumber
\end{align}
 is our first key step. Note that we have dropped out the constant term 
 $(1/(2\e)\|\bv_0\|^2$ 
 from $E_\e$ as it does not play a role in what follows, and that for aesthetic purposes we have changed $\bv_0$ to $-\bv_0$. 
The partial derivatives required in the statement of that theorem are
\begin{gather}
F_\bZ=\e\bZ+\bv''_0(t),\nonumber\\
 F_\bz=(\bF^\perp(\bu)\cdot\bz)\,\bF^\perp(\bu)+\e\bz+\bv'_0(t),\nonumber\\
F_\bu=(\bF^\perp(\bu)\cdot\bz)\,D\bF^\perp(\bu)\bz+\e\bu+\bv_0(t).\nonumber
\end{gather}
Equation \eqref{aaa} for  critical closed paths in $\espacioo$ for
$E_\e$ coming from Theorem \ref{basicsecond} involves the combination
\eqref{este} which in our case is
\begin{equation}\label{abscon}
\frac
d{dt}(\e\bu''_\e(t)+\bv''_0(t))-(\bF^\perp(\bu_\e(t))\cdot\bu'_\e(t))\,
\bF^\perp(\bu_\e(t))-\e\bu'_\e(t)-\bv'_0(t),
\end{equation}
which must be an absolutely continuous function in $[0, 1]$. Its
almost everywhere derivative ought to be, according to system
\eqref{aaa},
\begin{equation}\label{noder}
-(\bF^\perp(\bu_\e(t)\cdot\bu'_\e(t))\,
D\bF^\perp(\bu_\e(t))\bu'_\e(t)-\e\bu_\e(t)-\bv_0(t).
\end{equation}
Here 
$\bu_\e\in\espacioo$
is an arbitrary  critical closed path of
$E_\e$. In addition, from \eqref{saltoooe} we have
 \begin{equation}\label{ultsal}
[\e\bu''_\e(t)+\bv''_0(t)]_{t=0}=\bc.
\end{equation}
We need to examine these conditions carefully.

It is also important to stress how these conditions
ensure much more regularity for those  critical closed paths
precisely because they are  critical for a certain
functional. Even though paths in our ambient space are just in
$H^2([0, 1]; \R^2)$, critical closed paths of our family of functionals are much more regular. 
By Theorem \ref{basicsecond}, the
expression in \eqref{abscon} is absolutely continuous. Since the last
three terms of \eqref{abscon} and $\bv'''_0$ are continuous, we can
conclude that $\bu_\e$ is $\C^3$ in $[0, 1]$. Moreover, due to the
fact that the derivative of \eqref{abscon} is equal to \eqref{noder},
again by Theorem \ref{basicsecond}, it follows that $\bu_\e$ is even
$\C^4$ in $[0, 1]$ because all terms in \eqref{abscon}, when
differentiated with respect to $t$, are continuous except possibly
the first one $\bu'''_\e(t)$, and such a derivative is equal to
\eqref{noder} which is continuous. Note how condition \eqref{ultsal}
is redundant with the above information.

\begin{proposition}\label{dif}
Critical closed paths 
$\bu_\e\in\espacioo$
of our perturbed functional $E_\e$  are
$\C^\infty$ in $[0, 1]$, and are solutions of the system
\begin{gather}
\e(\bu''''_\e-\bu''_\e+\bu_\e) -\frac
d{dt}[(\bF^\perp(\bu_\e)\cdot\bu'_\e)\,\bF^\perp(\bu_\e)]+(\bF^\perp
(\bu_\e)\cdot\bu'_\e)\,
(\bu'_\e)^TD\bF^\perp(\bu_\e)\label{formafinal}\\
=-\bv''''_0+\bv''_0-\bv_0\nonumber
\end{gather}
in the interval $[0, 1]$.
\end{proposition}

\begin{proof}
Regularity conditions for $\bu_\e$ have been discussed in the paragraph prior to the
statement of the proposition. Equation \eqref{formafinal} is a
consequence, according to equation \eqref{aaa}, of expressing the
equality of the derivative of \eqref{abscon} with \eqref{noder}. A typical bootstrap argument yields the regularity claimed in the statement.
\end{proof}

System \eqref{formafinal} is a key point for counting the critical
closed paths of the functional $E_\e$. We are facing a singularly-perturbed, fourth-order ODE system \eqref{formafinal} with periodic
(unknown) boundary conditions. Our plan to count, and eventually find
an upper bound for, the number of branches of solutions of \eqref{formafinal} proceeds in
two steps:
\begin{enumerate}
\item for a fixed such branch, understand its asymptotic behavior as $\e\searrow0$, to count how many such different asymptotic behaviors there might be; and
\item for a fixed such asymptotic behavior, decide how many branches may converge to it.
\end{enumerate}

As in Definition \ref{solutionmap}, we introduce the following.
\begin{definition}
The mapping
$$
\bu_\e(t; \bp_0, \bp_1, \bp_2, \bp_3):[0, 1]\times\R^2\times\R^2\times\R^2\times\R^2\to\R^2
$$
will designate the solution mapping for problem \eqref{formafinal} under initial conditions
$$
\bu^{i)}_\e(0; \bp_0, \bp_1, \bp_2, \bp_3)=\bp_i,\quad i=0, 1, 2, 3.
$$
\end{definition} 
According to the shooting method, and, in particular as a consequence of Corollary \ref{importantez}, the following definition is legitimate.

\begin{definition}\label{mapasol}
Let 
$(\bp_0, \bq_0)\in\R^2\times\R^2$
be a given pair. For $(\bp, \bq)$ in a neighborhood of $(\bp_0, \bq_0)$,
$\bu(t; \bp, \bq, \e)$ will designate the unique periodic solution of system \eqref{formafinal} such that
$$
\bu(0; \bp, \bq, \e)=\bu(1; \bp, \bq, \e)=\bp, \quad \bu'(0; \bp, \bq, \e)=\bu'(1; \bp, \bq, \e)=\bq.
$$
This map $\bu$ is analytic in all its variables. 
\end{definition}
To check that this definition is meaningful, after the above-mentioned corollary, 
simply note that the perturbed integrand 
\begin{align}
F(t, \bu, \bz, \bZ)=&\frac12(\bF^\perp(\bu)\cdot\bz)^2+\frac\e2(|\bZ|^2+|\bz|^2+|\bu|^2)\nonumber\\
&+\bu\cdot\bv_0(t)+\bz\cdot\bv'_0(t)+\bZ\cdot\bv''_0(t)\nonumber
\end{align}
for $\e$ fixed, is analytic in all its variables provided the auxiliary path $\bv_0(t)$ can be chosen analytic as well. This is no restriction, since there is plenty of freedom to select $\bv_0$ from: it can be chosen analytic, and uniformly bounded (see Theorem \ref{central}).

\subsection{Asymptotic behavior}\label{ab}
For the sake of transparency, and to facilitate a few interesting
computations, we recast system \eqref{formafinal} in its two components
$$
\begin{array}{r}
(ZQ)'+Z(-Q_xx'+P_xy')=-\e\alpha_1,\\
(ZP)'+Z(Q_yx'-P_yy')=\e\alpha_2.
\end{array}
$$
where
\begin{gather}
\bF=(P, Q), \quad \bu_\e=(x, y),\quad \bv_0=(X, Y)\nonumber\\
Z\equiv\bF^\perp(\bu_\e)\cdot\bu'_\e=P(x, y)y'-Q(x, y)x',\nonumber\\
W\equiv\bF(\bu_\e)\cdot\bu'_\e=P(x, y)x'+Q(x, y)y',\nonumber\\
\Dv\equiv P_x+Q_y,\quad \alpha_1=\overline x''''-\overline
x''+\overline x,\quad
 \alpha_2=\overline y''''-\overline y''+\overline y,\nonumber
\end{gather}
with 
$\overline x=x+X$, $\overline y=y+Y$.
Note that $Z^2/2$ is
precisely the integrand for $E_0$, and recall that all close paths
involved are $1$-periodic, $\C^\infty$ and belong to $\bbO$, so
that we can freely differentiate in $t$ as many times as needed. In
particular, the two equations of the system of  critical closed paths
become
\begin{equation}\label{primera}
Z'Q+Z\Dv y'=-\e\alpha_1,\quad Z'P+Z\Dv x'=\e\alpha_2.
\end{equation}

We manipulate the two equations in \eqref{primera} in two ways:
\begin{enumerate}
\item  multiply the first equation by $Q$, the second by $P$, and add
up the results to find
\begin{equation}\label{aprimera}
Z'(P^2+Q^2)=-\e(\alpha_1 Q-\alpha_2 P)-ZW\Dv;
\end{equation}
\item then, multiply the first by $P$, the second by $Q$, and subtract
the results to have
$$
Z^2\Dv=-\e(\alpha_1 P+\alpha_2 Q).
$$
\end{enumerate}
We remind readers that we are
searching for periodic solutions of this system for which $E_\e$ is away from
zero. Indeed, recall that in seeking an upper bound for the right-hand side in \eqref{cotapreliminar}, $a$ is small, but kept fixed, when computing the inner limit
$\lim_{\e\to0}M_{cri, \e, a}$.
The second part of Theorem \ref{ultimoz} ensures that there cannot be branches of critical paths $\{\bu_\e\}$ with $E_0(\bu_\e)\to0$. 
We will therefore discard from our consideration those
critical paths $\bu_\e$ for which $E_0$ is arbitrarily small. In particular, we do
not need to consider asymptotic behaviors reducing to a point, and
so, bearing in mind that equilibria of our polynomial, differential
system are isolated and they could only be associated with critical
closed paths of the kind we are not interested in (those with small value for $E_0$), we can further multiply
\eqref{aprimera} by $Z$ and divide by $P^2+Q^2$, to have, taking into
account the other equation,
$$
(Z^2)'=2\e(\alpha_1 x'+\alpha_2 y').
$$
Hence, system \eqref{formafinal} can be written in the simplified,
equivalent form
$$
(Z^2)'=2\e(\alpha_1 x'+\alpha_2 y'),\quad Z^2\Dv=-\e(\alpha_1
P+\alpha_2 Q).
$$
To avoid confusion, we will rather write
\begin{equation}\label{sistema}
(Z_\e^2)'=2\e(\alpha_1 x'_\e+\alpha_2 y'_\e),\quad Z_\e^2\Dv_\e=-\e(\alpha_1
P_\e+\alpha_2 Q_\e).
\end{equation}
to stress the dependence on $\e$ of all quantities. Recall that
$$
Z_\e=\bF^\perp(\bu_\e)\cdot\bu'_\e,\quad \Dv_\e=\dv\bF(x_\e, y_\e),\quad \bu_\e=(x_\e, y_\e).
$$
The first point in Theorem \ref{ultimoz} informs us that in fact, along branches of critical paths $(x_\e, y_\e)$, we can neglect terms multiplied by $\epsilon$, and
have that
$$
(Z_\e^2)'\to0,\quad Z_\e^2\Dv_\e\to0
$$
in $L^2(0, 1)$. We record this fact in a formal statement.

\begin{lemma}\label{conacero}
Let $(x_\e, y_\e)$ be a branch of solutions of \eqref{sistema}.
For a suitable subsequence (not relabeled), 
$$
(Z_\e^2)'\to0,\quad Z_\e^2\Dv_\e\to0,
$$
in $L^2(0, 1)$ and pointwise for a.e. $t\in[0, 1]$.
\end{lemma}
Note how the $L^2$-convergence claimed in this statement forbids that $Z_\e^2$ could converge to a non-constant function. 
We are therefore entitled to understand all possible asymptotic behaviors of critical closed paths 
$(x_\e, y_\e)\in\bbO$
through an analysis of the limit system  
\begin{equation}\label{sistemalimite}
(Z^2)'=0,\quad Z^2\Dv=0,
\end{equation}
setting $\e=0$ in \eqref{sistema}.
The first equation in \eqref{sistemalimite} implies that $Z^2=k^2$ (note that here is crucial the $L^2$ -convergence in Lemma \ref{conacero}),
but since we are only interested in the asymptotic behavior for
critical closed paths whose value for $E_0$ stays away from zero, we
discard the case $k=0$. In this situation, the second equation in
\eqref{sistemalimite}, implies $\Dv=0$. We would like to understand
the nature of solutions of the limit system
\begin{equation}\label{zeta}
Z^2=k^2>0,\quad \Dv=0.
\end{equation}
We write this system in the form, differentiating the second
equation,
\begin{equation}\label{dsform}
\bF^\perp(\bu)\cdot\bu'=\pm k\neq0,\quad \nabla\Dv(\bu)\cdot\bu'=0.
\end{equation}
This is an implicit, first-order system that becomes singular when
the determinant 
$\nabla\Dv(\bu)\cdot\bF(\bu)$ 
of the matrix of the system
$$
\begin{pmatrix}\bF^\perp(\bu)\\ \nabla\Dv(\bu)\end{pmatrix}
$$
vanishes. These singular points are precisely the
contact points of our differential system over the curve $\Dv=0$. The
fact that
$$
\bF^\perp(\bu)\cdot\bu'=\pm k\neq0
$$ 
shows that $\bu_\e$, for $\e$ sufficiently small, can only turn
around, changing $+k$ by $-k$ or viceversa, near those contact
points.  As a matter of fact, critical closed paths $\bu_\e$ do have to
turn around whenever one such point is reached and it is a single root of system \eqref{salto0}. To emphasize this point, we explicitly introduce the following.

\begin{definition}
A point $\bp\in\R^2$ is a single (or simple) contact point of our differential system, if
$$
\left.\det\begin{pmatrix}\bF^\perp(\bu)\\ \nabla\Dv(\bu)\end{pmatrix}\right|_{\bu=\bp}=0
$$
and it changes sign in every neighborhood of $\bp$ in $\Dv=0$.
\end{definition}

We can now establish in a precise way the role played by single, contact points.
\begin{lemma}\label{conauno}
Let $\bu_\e=(x_\e, y_\e)$ be a branch of critical closed paths of $E_\e$ such that $E_0(\bu_\e)$ stays away from zero. 
Suppose that in a certain subinterval  
$$
[t^-_\e, t^+_\e]\subset[0, 1],\quad t_\e^+-t_\e^-\to0,
$$ 
we know that
$$
|\bu'_\e(t^\pm_\e)|\to+\infty, \quad \bu_\e(t^\pm_\e)\to\bp,
$$
as $\e\to0$, where $\bp$ is a single, contact point of the system. Then $\bu_\e$ must turn around at $\bp$ for $\e$ sufficiently small, in the sense
\begin{equation}\label{conclusion}
\frac{\bu'_\e(t^-_\e)}{|\bu'_\e(t^-_\e)|}+\frac{\bu'_\e(t^+_\e)}{|\bu'_\e(t^+_\e)|}\to\bc
\end{equation}
as $\e\to0$.
\end{lemma}

\begin{proof}
Argue by contradiction. Suppose that $\bu_\e$ does not turn around as it approaches a single contact point $\bp$. By \eqref{dsform}, the product 
\begin{equation}\label{producto}
\bF^{\perp}(\bu)\cdot\bu'(=Z)
\end{equation}
should retain its sign  in a vicinity of $\bp$ to avoid a jump discontinuity if $\bu$ is the corresponding limit path. In addition $\bu'$ should be continuous across $\bp$ as well. If we keep in mind that $\nabla\Dv(\bu)\cdot\bu'=0$ at $\bp$, it is easy to see that these facts are not compatible with the change of sign at $\bp$ of the determinant in the definition of a single contact point for that a change of sing of such determinant would imply a change of sign of the product in \eqref{producto}. Note how all these facts together would not lead to a contradiction if $\bp$ is a high-order contact point. 
\end{proof}

Our above discussion can be summarized in the following statement
that classifies all possible asymptotic behaviors for critical closed
paths.

\begin{theorem}\label{recuentoo}
Assume that all the components of the curve $\Dv=0$ are topologically
straight lines or ovals. The possible limit behaviors as $\e\to 0$ of branches of critical closed
paths of $E_\varepsilon$, with critical value uniformly away from zero, can be 
identified with arcs of the connected
components of the curve $\Dv=0$ in one of the following
possibilities:
\begin{itemize}
\item[(a)] If the component is homeomorphic to a straight line, then
\begin{itemize}
\item[(a.1)] the limit behavior is an arc whose endpoints are two contact
points and no additional contact point lies in its interior;
\item[(a.2)] the limit behavior is an arc whose endpoints are one contact
point and the infinity, and no additional contact point can be found in its
interior;
\item[(a.3)] the limit behavior is the whole component without contact points.
\end{itemize}
\item[(b)] If the component is homeomorphic to an oval, then
\begin{itemize}
\item[(b.1)] the limit behavior is an arc whose endpoints are two contact
points and no additional contact point lies in its interior;
\item[(b.2)] the limit behavior is an arc covering the full oval whose
endpoints have to be  a single contact point, and no additional
contact point is to be found in the oval;
\item[(b.3)] the limit behavior is the whole oval without contact points.
\end{itemize}
\end{itemize}
\end{theorem}

\subsection{Multiplicity}\label{mult}
We are concerned in this subsection about the possibility that various
branches of the set of  critical closed paths, for $\epsilon$
positive, may coalesce into the same limit behavior as $\e\searrow0$,
and how they can possibly contribute to the global 
number of critical paths. We will need to exploit very precisely the information that the number of any such group of critical paths share, at least locally, the same limit behavior as $\e\to0$. Fix, therefore, any such asymptotic behavior of the collection of possibilities given in Theorem \ref{recuentoo}. 

As a consequence of Definition \ref{mapasol}, and comments made after it, the set of possible singular pairs $(\bp, \bq)$ where maps 
$\bu(t; \bp, \bq, \e)$
are not well-defined because there might be more than one critical path with
\begin{equation}\label{estrellaz}
\bu(0; \bp, \bq, \e)=\bu(1; \bp, \bq, \e)=\bp,\quad 
\bu'(0; \bp, \bq, \e)=\bu'(1; \bp, \bq, \e)=\bq,
\end{equation}
is isolated, for each fixed $\e$, as a consequence of analyticity. This is in fact the only reason to stress the analytic dependence of this map on its variables. 
For a given (countable) sequence of values for $\e$, we can always find a certain point $\bp$ in the part of the curve $\Dv=0$, with tangent vector $\bq$, identifying the particular asymptotic limit we are focusing on, and away from any contact point of the system limiting the asymptotic behavior that is being considered according to Theorem \ref{recuentoo}, in such a way that the family of maps
\begin{equation}\label{aplicacion}
\bu(t; \by, \bz, \e):[0, 1]\times{\mathbb B}_\rho(\bp)\times{\mathbb B}_\rho(\bq)\to\R^2
\end{equation}
is well-defined, smooth (even analytic) in all variables $(t, \by, \bz)$, for every  value of $\e$ from the selected sequence, and some positive radius $\rho$. The interesting point is that the domain of this family of maps can be taken independently of $\e$, at least for a sequence of values of $\e$ converging to zero. We will not make the distinction between a countable set of values of $\e$ and the full continuum since this distinction is irrelevant for our purposes. 
Recall that, in addition to verifying \eqref{formafinal} in Proposition \ref{dif}, we also have
$$
\bu(0; \by, \bz, \e)=\bu(1; \by, \bz, \e)=\by,\quad 
\bu'(0; \by, \bz, \e)=\bu'(1; \by, \bz, \e)=\bz,
$$
for 
$$
(\by, \bz)\in{\mathbb B}_\rho(\bp)\times{\mathbb B}_\rho(\bq).
$$
We are going to proceed in two successive steps. 
\begin{enumerate}
\item Suppose $\by$ is kept fixed in a neighborhood of $\bp$. How many critical paths of the perturbed functional $E_\e$ in Proposition \ref{dif} could there be starting at $\by$? We will show that there cannot be more than $n-1$ such paths regardless of the value of $\by$. These $n-1$ paths will correspond to our mapping $\bu(t; \by, \bz, \e)$ for certain suitable velocity vectors 
$\bz=\bz_i(\by)$, $i=1, 2, \dots, n-1$,
that is to say
\begin{equation}\label{segundopaso}
\bu_i(t; \by, \e)\equiv\bu(t; \by, \bz_i(\by), \e),\quad i=1, 2, \dots, n-1,
\end{equation}
are the critical paths in Proposition \ref{dif} for each fixed $\by$. The maps $\bz_i(\by)$ are again smooth due to the standard regularity dependence of solutions of differential problems on initial conditions. 
\item For each given $i=1, 2, \dots, n-1$, how many values of $\by$ in a neighborhood of $\bp$ can there be such that $\bu_i(t; \by, \e)$ in \eqref{segundopaso} is one of the critical paths we are interested in counting, i. e. critical paths in Proposition \ref{dif}? Again we will show that, for each $i$, there are at most $n-1$ such paths. 
\end{enumerate}

Our strategy is led by two main points.
\begin{enumerate}
\item We would like to replace vector variables $\by$ and $\bz$ in \eqref{aplicacion} by single variables $r$ and $s$, respectively, since our arguments will depend on facts that are only valid for polynomials of a single real variable.
\item Use the full power of being a critical path to localize arguments through suitable test paths, and reduce the counting procedure to a question about one-variable polynomials. 
\end{enumerate}

Since paths $\bu(t; \by, \bz, \e)$ will intersect the straight line 
$\{\bp+r\bq^\perp: r\in\R\}$,
at least for small $\e$ and small $r$, due to the convergence to one of our asymptotic limits, by resetting the initial time $t=0$ and reparametrizing the path conveniently around $t=0$, we can replace $\by$ in some neighborhood of $\bp$ by $\bp+r\bq^\perp$, for small $r$, in the above discussion in such a way that we have a single variable $r$ to label the starting point of our critical paths. To make things easier, we will replace the unit interval $[0, 1]$ by $[-1/2, 1/2]$ in this section. In the same vein, in a neighborhood of $\bq$, the velocity vector $\bz$ can be considered of the form $\bq+s\bq^\perp$ for $|s|$ sufficiently small, possibly at the expense of changing slightly, through a linear reparametrization to transform appropriately the length of $\bz$ and without changing the feature of being critical for $E_\e$, the unit interval $[-1/2, 1/2]$ by $[-S_\e, S_\e]$ with $S_\e\to1/2$ as $\e\to0$. This is not a problem as we will subsequently work in subintervals $J_\e$ of small length around $0$. In this way, we can assume, without loss of generality in our above discussion, that
\begin{gather}
\by=\bp+r\bq^\perp,\quad \bz=\bq+s\bq^\perp,\nonumber\\
\bu(t; r, s, \e)\equiv\bu(t; \bp+r\bq^\perp, \bq+s\bq^\perp, \e),\nonumber
\end{gather}
in suitable respective neighborhoods of $\bp$ and $\bq$, 
and so the two steps in our counting method will depend in both cases on single variables either $r$ or $s$. More specifically, we can review the preceding two main successive stages in the following explicit manner.
\begin{enumerate}
\item Assume $r$ is given of size sufficiently small. How many values of $s$ could there possibly be so that $\bu(t; r, s, \e)$ represent critical paths in Proposition \ref{dif}? Argue that there cannot be more than $n-1$, corresponding to the $n-1$ smooth branches $s_i(r)$. 
\item For each such branch $s_i(r)$, show that there cannot possibly be more than $n-1$ values of $r$ such that 
\begin{equation}\label{secondstep}
\bu_i(t; r, \e)=\bu(t; r, s_i(r), \e),\quad i=1, 2, \dots, n-1,
\end{equation}
are critical paths in Proposition \ref{dif}, i.e. the paths we aim at counting. 
\end{enumerate}
If we can cover these two steps successfully, we will have our main result in this section. It may be interesting to realize that if some of the roots for $r$ or $s$ that we pretend to count correspond to the same critical path (because the path goes back and forth for instance), the bound we are claiming would still be valid for in that case there would be less distinct critical paths. 

\begin{theorem}\label{multiplicidad}
There cannot be more than $(n-1)^2$ branches of critical closed paths in Proposition \ref{dif} converging to any of the possible asymptotic behaviors given in  Theorem \ref{recuentoo}.
\end{theorem}
\begin{proof}
As just indicated, we will proceed in two main successive steps: first, for fixed $r$, try to understand critical paths $\bu(t; r, s, \e)$ in Proposition \ref{dif} depending on $s$; secondly, for each branch $i$ found in the previous step, make an attempt to examine critical paths $\bu_i(t; r, \e)$, given in \eqref{secondstep}, in Proposition \ref{dif}. In both cases, critical paths 
$\bu(r, s, \e)\equiv \bu(t; r, s, \e)$
make the derivative $E'_\e$ vanish, through legitimate test paths 
$\bW_\e\in\espacioo$,
that will be appropriately chosen below in each situation. More explicitly, the derivative of $E_\e$ in the direction $\bW_\e$, given by
\begin{equation}\label{derada}
\langle E'_\e(\bu(r, s, \e)), \bW_\e\rangle=
\langle E'_0(\bu(r, s, \e)), \bW_\e\rangle+\epsilon\langle\bu(r, s, \e),\bW_\e\rangle+\langle\bv_0, \bW_\e\rangle
\end{equation}
should vanish on critical paths. It is important to realize that such test paths $\bW_\e$ can depend on $\e$, since this will permit a flexibility that is very convenient to our purposes. 
As already remarked before, our strategy is to localize the vanishing of \eqref{derada} through a judicious choice for $\bW_\e$ so as to reduce the counting procedure essentially to an issue of roots of polynomials of a single variable. 
Note that $\langle\cdot, \cdot\rangle$ stands for the inner product in $\espacio$, i.e. 
integration with respect to the time variable $t$ of products of derivatives up to order two. 

The fundamental difference between those test paths $\bW_\e$ in the two situations we would like to consider is related to the fact that in the first case $\bW_\e$ must vanish for $t=0$ (because the starting point for competing paths is taken as the constant vector $\bp+r\bq^\perp$), while for the second they don't have to.
Recall that 
$$
\bu(t; r, s, \e)\equiv\bu(t; \bp+r\bq^\perp, \bq+s\bq^\perp, \e),\quad \bu(r, s, \e)=\bu(t; r, s, \e), 
$$
and that the test path $\bW_\e$ can also depend eventually on $\e$, as already emphasized. Before we reach the point where the difference for the test path $\bW_\e$ must be taken into account, let us examine the explicit form of the derivative in \eqref{derada}. 

If we recall formula \eqref{prider}, we can write explicitly the term
$\langle E'_0(\bu(r, s, \e)), \bW_\e\rangle$
in the form
\begin{gather}
\int_0^1(\bF^\perp(\bu(r, s, \e))\cdot\bu'(r, s, \e))
(D\bF^\perp(\bu(r, s, \e))\bW_\e)\cdot \bu'(r, s, \e)\,dt\label{estrella}\\
+\int_0^1(\bF^\perp(\bu(r, s, \e))\cdot\bu'(r, s, \e))
\bF^\perp(\bu(r, s, \e))\cdot\bW'_\e]\,dt,\nonumber
\end{gather}
where primes indicate differentiation with respect to time $t$. We will eventually take
\begin{equation}\label{eleccion}
\bW_\e(t)=\chi_\e(t)\tilde \bW_\e(t)
\end{equation}
with $\{\chi_\e\}$, a sequence of suitable mollifiers (recall Subsection \ref{multiplicidadd}), and $\tilde \bW_\e(t)$, chosen later appropriately.  
Since the support $J_\e$ of the mollifier $\chi_\e$ can be selected in such a way that the factor 
$$
\bF^\perp(\bu(r, s, \e))\cdot\bu'(r, s, \e),
$$
which is the integrand for $E_0$, tends to be a constant $k_\e$ uniformly away from zero, according to our discussion in  Section \ref{ab}, and, after all, we are interested in values of pairs $(r, s)$ vanishing the derivative in \eqref{derada}, we realize that we can replace the previous form of the derivative (whose first term is written in \eqref{estrella}) by the expression
\begin{gather}
k_\e\int_{J_\e}[(D\bF^\perp(\bu(r, s, \e))\bW_\e)\cdot \bu'(r, s, \e)+\bF^\perp(\bu(r, s, \e))\cdot\bW'_\e]\,dt\nonumber\\
+\epsilon\langle\bu(r, s, \e),\bW_\e\rangle+\langle\bv_0, \bW_\e\rangle.\nonumber
\end{gather}
Let us focus our attention on the integral
$$
I_\e=\int_{J_\e}[(D\bF^\perp(\bu(r, s, \e))\bW_\e)\cdot \bu'(r, s, \e)+\bF^\perp(\bu(r, s, \e))\cdot\bW'_\e]\,dt.
$$
We can perform an integration by parts in the second term (note that contributions at end-points vanish) to find
$$
I_\e=\int_{J_\e}
[(D\bF^\perp(\bu(r, s, \e))\bW_\e)\cdot \bu'(r, s, \e)-
(D\bF^\perp(\bu(r, s, \e))\bu'(r, s, \e))\cdot \bW_\e]\,dt
$$
or, by recalling \eqref{eleccion}, 
\begin{align}
I_\e=&\int_{J_\e} \chi_\e(t)[D\bF^\perp(\bu(r, s, \e))^T-D\bF^\perp(\bu(r, s, \e))]\bu'(r, s, \e)\cdot \tilde\bW_\e\,dt\nonumber\\
=&-\int_{J_\e}\chi_\e(t)\Dv(\bu(r, s, \e)) \bu'(r, s, \e)^\perp\cdot \tilde\bW_\e\,dt,\nonumber
\end{align}
where
$$
\bu'(r, s, \e)^\perp=\bQ\bu'(r, s, \e),\quad \bQ=\begin{pmatrix}0&-1\\1&0\end{pmatrix}.
$$
The family of functions $f_\e(r, s)$ of two variables $(r, s)$ we need to deal with is then, after an irrelevant change of sign, 
\begin{align}
f_\e(r, s)\equiv &\int_{J_\e}\chi_\e(t)\Dv(\bu(r, s, \e)) \bu'(r, s, \e)^\perp\cdot \tilde\bW_\e\,dt\label{funciones}\\
&-\frac{\epsilon}{k_\e}\langle\bu(r, s, \e),\bW_\e\rangle-\frac1{k_\e}\langle\bv_0, \bW_\e\rangle.\nonumber
\end{align}

We now focus on the first step of our program. Fix $r$, small. For a family of critical paths $\bu(r, s, \e)$, the family of functions in \eqref{funciones} should vanish according to our above discussion, and so a bound for the number of such roots in the variable $s$ will furnish an upper bound for the number of such critical paths $\bu(r, s, \e)$. 
Recall that test path $\bW_\e$ is the product in \eqref{eleccion}.
It suffices to select the support of the mollifier in the form $J_\e=[0, 2\alpha(\e)]$
to ensure that
$$
\bW_\e(0)=\chi_\e(0)\tilde\bW_\e(0)=\bcero.
$$
This is guaranteed by the previous choice for the support of $\chi_\e$. 
Other than that, $\tilde\bW_\e(t)$ can be chosen in an arbitrary way to our advantage, as well as $\alpha(\e)>0$. 

On the other hand, consider the family of polynomials of a single variable $s$, for each fixed $r$, given by
\begin{equation}\label{polinomioss}
P_\e(s)=\Dv[\bp+r\bq^\perp+\alpha(\e)(\bq+s\bq^\perp)].
\end{equation}
Keep in mind that $\Dv(x, y)$ is a polynomial of at most degree $n-1$ in two variables. 
Let $N\le n-1$ be the degree of such polynomials in the variable $s$ (recall that $r$ is being kept frozen here), and take 
 \begin{equation}\label{estrellaaa}
\tilde\bW_\e(t)\equiv\frac1{\alpha(\e)^N}\bq^\perp, \hbox{ constant},\quad \bW_\e=\chi_\e(t)\tilde\bW_\e.
\end{equation}
That degree $N$ will play an important role in what follows.

Assume that, for fixed $r$, we could find at least $N+1$ branches of values for $s$ that are roots of $f_\e$ in \eqref{funciones}, including possibly multiplicities, i.e. 
$$
f_\e(r, s_j(\e, r))=0,\quad s=s_j(\e, r), j=1, 2, \dots, N, N+1.
$$
Because each $f_\e$ is smooth, there would be a certain value $s(\e)$, $s(\e)\to0$, such that
$$
\frac{\partial^N f_\e}{\partial s^N}(r, s(\e))=0\hbox{ for all }\e.
$$
We have dropped the dependence of $s$ on $r$ here for the sake of notational simplicity. 
If we go back to our formula for $f_\e(r, s)$, and take into account that all functions involved are smooth and variables move in intervals where everything is bounded, we can take differentiation under the integral sign, and find that
\begin{align}
0=&\int_{J_\e}\chi_\e(t)\left.\frac{\partial^N }{\partial s^N}\right|_{s=s(\e)}[\Dv(\bu(r, s, \e)) \bu'(r, s, \e)^\perp\cdot \tilde\bW_\e]\,dt\nonumber\\
&-\frac\e{k_\e}\left.\frac{\partial^N }{\partial s^N}\right|_{s=s(\e)}\langle\bu(r, s, \e),\bW_\e\rangle.\nonumber
\end{align}
For the second term, we can write, after several integrations by parts, 
\begin{align}
\langle\bu(r, s, \e),\bW_\e\rangle=&\int_{J_\e}\sum_{i=0}^2\bu^{i)}(t; r, s, \e)\cdot\bW_\e^{i)}(t)\,dt\nonumber\\
=&\int_{J_\e}\chi_\e(t)\sum_{i=0}^2(-1)^i\bu^{2i)}(t; r, s, \e)\cdot\tilde\bW_\e(t)\,dt.\nonumber
\end{align}
By the second part of Proposition \ref{aclarar}, if $\alpha(\e)$ is chosen sufficiently small, we can conclude that
\begin{gather}
\left.\frac{\partial^N }{\partial s^N}\right|_{s=s(\e)}[\Dv(\bu(\alpha(\e); r, s, \e)) \bu'(\alpha(\e); r, s, \e)^\perp\cdot \tilde\bW_\e]\label{importantisimo}\\
-\frac\e{k_\e}\left.\frac{\partial^N }{\partial s^N}\right|_{s=s(\e)}\sum_{i=0}^2(-1)^i\bu^{2i)}(\alpha(\e); r, s, \e)\cdot\tilde\bW_\e\to0,\nonumber
\end{gather}
as $\e\searrow0$. 
We would like to argue that this vanishing limit contradicts the choice of the degree for polynomials $P_\e(s)$ made in \eqref{polinomioss}. 

Again, because of smoothness,
\begin{gather}
\bu(\alpha(\e); r, s, \e))=\bu(0; r, s, \e)+\alpha(\e)\bu'(0; r, s, \e)+\alpha(\e)^2 \bR(t; r, s, \e),\nonumber\\
\bu'(\alpha(\e); r, s, \e))=\bu'(0; r, s, \e)+\alpha(\e)\bu''(0; r, s, \e)+\alpha(\e)^2 \bR'(t; r, s, \e)\nonumber
\end{gather}
with bounded remainders 
$\bR(t; r, s, \e)$, $\bR'(t; r, s, \e)$ 
which are analytic functions of all their variables in the domain where those variables move. Even more explicitly, because of \eqref{estrellaz}, 
\begin{gather}
\bu(\alpha(\e); r, s, \e))=\bp+r\bq^\perp+\alpha(\e)(\bq+s\bq^\perp)+\alpha(\e)^2 \bR(t; r, s, \e),\nonumber\\
\bu'(\alpha(\e); r, s, \e))=(\bq+s\bq^\perp)+\alpha(\e)\bu''(0; r, s, \e)+\alpha(\e)^2 \bR'(t; r, s, \e).\nonumber
\end{gather}
If we take these expressions into the partial derivatives in \eqref{importantisimo}, and recall \eqref{estrellaaa}, 
we see that the first term becomes
\begin{gather}
\left.\frac{\partial^N }{\partial s^N}\right|_{s=s(\e)}\frac1{\alpha(\e)^N}[\Dv(\bp+r\bq^\perp+\alpha(\e)(\bq+s\bq^\perp)+\alpha(\e)^2 \bR(t; r, s, \e))]\times\nonumber\\
[(\bq+s\bq^\perp)^\perp+\alpha(\e)\bu''(0; r, s, \e)^\perp+\alpha(\e)^2 \bR'(t; r, s, \e)^\perp]\cdot \bq^\perp\nonumber
\end{gather}
or
\begin{gather}
\left.\frac{\partial^N }{\partial s^N}\right|_{s=s(\e)}\frac1{\alpha(\e)^N}[\Dv(\bp+r\bq^\perp+\alpha(\e)(\bq+s\bq^\perp)+\alpha(\e)^2 \bR(t; r, s, \e))]\times\nonumber\\
[|\bq|^2+(\alpha(\e)\bu''(0; r, s, \e)^\perp+\alpha(\e)^2 \bR'(t; r, s, \e)^\perp)\cdot \bq^\perp].\nonumber
\end{gather}
Because of the way the polynomials $P_\e$ and their degree $N$ were selected in \eqref{polinomioss}, on the one hand; and the freedom we have to choose $\alpha(\e)$ converging to zero as rapidly as necessary, on the other, we see that the previous derivative will not vanish as $\e\to0$ since terms affected by a negative power of $\alpha(\e)$ will vanish when differentiation is performed, while the coefficient corresponding to power $s^N$ does not vanish, by our manner of choosing the degree $N$ in \eqref{polinomioss}, and is independent of $\alpha(\e)$ (and of $\e$). The terms affected by the remainders $\bR$ will vanish too because of the presence of the higher power $\alpha(\e)^2$.
The second term in \eqref{importantisimo} does not spoil our argument because of the presence of $\e$ in front of it. This is easily checked. Note how the reason for the choice in \eqref{estrellaaa} is to have that the additional variable $s$ in $(\bq+s\bq^\perp)^\perp\cdot\bq^\perp$ drops out. 

This contradiction between \eqref{importantisimo} and polynomials in \eqref{polinomioss} enables us to conclude that there cannot be more than $n-1$ branches of solutions
$f_\e(r, s_j(\e, r))=0$ for $j=1, 2, \dots, n-1$.
This is the end of the first step. 

For the second step, fix one of those $n-1$ branches
$(r, s_j(r, \e))$ for $1\le j\le n-1$,
and put
\begin{equation}\label{caminosdef}
\bu_j(t; r, \e)\equiv\bu(t; r, s_j(r, \e), \e),\quad \bu_j(r, \e)=\bu_j(t; r, \e),
\end{equation}
for $t\in[-1/2, 1/2]$, and small $r$ and $\e$. We go back to the family of functions in \eqref{funciones} setting $s=s_j(r, \e)$, to define the smooth functions
\begin{equation}\label{casoj}
f_{j, \e}(r)=\int_{J_\e}\chi_\e(t)\Dv(\bu_j(r, \e)) \bu'_j(r, \e)^\perp\cdot \tilde\bW_\e\,dt-\frac{\epsilon}{k_\e}\langle\bu_j(r, \e),\bW_\e\rangle-\frac1{k_\e}\langle\bv_0, \bW_\e\rangle.
\end{equation}
This time however, since we are letting variable $r$ change, the value of test paths 
$\bW_\e=\chi_\e\tilde\bW_\e$
need not vanish at $t=0$, and hence mollifier $\chi_\e$ can be taken to be even and with support $[-\alpha(\e), \alpha(\e)]$ with $\alpha(\e)$, as in the first step, at our disposal. Our goal is, again, to argue that for each fixed $j$, the family of functions $f_{j, \e}(r)$ in \eqref{casoj} cannot have more than $n-1$ branches of roots as $\e\to0$. If we succeed in showing this, our claim in the theorem will be proved. 

Our job, after the first step, is now much easier. We follow exactly along the same lines. Consider the polynomial
$P(r)=\Dv(\bp+r\bq^\perp)$
whose degree $N$ is at most $n-1$. Take, almost as above,
 $$
\tilde\bW_\e(t)\equiv\bq^\perp,\quad \bW_\e=\chi_\e(t)\tilde\bW_\e.
$$
Suppose that there could be $N+1$ roots $r_{i, j}(\e)$, converging to zero as $\e\to0$, of $f_{j, \e}$, i.e.
$$
f_{j, \e}(r_{i, j}(\e))=0,\quad i=1, 2, \dots, N, N+1.
$$
Recall that $j$ is fixed now. Due to smoothness, there should be, at least, one root $r_j(\e)$ of the $N$-th derivative
$$
f^{N)}_{j, \e}(r_j(\e))=0,\quad r_j(\e)\to0\hbox{ as }\e\to0.
$$
Through \eqref{casoj}, we can write
\begin{gather}
\int_{J_\e}\chi_\e(t)\left.\frac{d^N}{dr^N}\right|_{r=r_j(\e)}\Dv(\bu_j(r, \e)) \bu'_j(r, \e)^\perp\cdot \tilde\bW_\e\,dt\nonumber\\
-\frac{\epsilon}{k_\e}\left.\frac{d^N}{dr^N}\right|_{r=r_j(\e)}\langle\bu_j(r, \e),\bW_\e\rangle=0.\nonumber
\end{gather}
We focus on the important first term, since the second one is tackled as in the first step through the first part of Proposition \ref{aclarar}, to conclude that its limit, as $\e\to0$, vanishes as well. In this case, by selecting $\alpha(\e)$ appropriately, we conclude that
$$
\left.\frac{d^N}{dr^N}\right|_{r=r_j(\e)}\Dv(\bu_j(0; r, \e)) \bu'_j(0; r, \e)^\perp\cdot \bq^\perp\to0,
$$
as $\e\searrow0$. 
But, taking into account \eqref{caminosdef}, 
$$
\bu_j(0; r, \e))=\bp+r\bq^\perp,\quad \bu'_j(0; r, \e))=\bq+s_j(r, \e)\bq^\perp,
$$
and the previous derivative becomes
$$
\left.\frac{d^N}{dr^N}\right|_{r=r_j(\e)}\Dv(\bp+r\bq^\perp) (\bq+s_j(r, \e)\bq^\perp)^\perp\cdot \bq^\perp=|\bq|^2\left.\frac{d^N}{dr^N}\right|_{r=r_j(\e)}P(r).
$$
The condition
$$
\left.\frac{d^N}{dr^N}\right|_{r=r_j(\e)}P(r)\to0
$$
as $\e\to0$ is impossible for a true polynomial of degree $N$. This is the contradiction, and the end of our proof.
\end{proof}

\section{Counting method}\label{contar}
Once we have proved all of the main preceding steps, it is easy to describe how to organize the counting procedure for an upper bound for $M_{cri, \e, a}$ independent of $a$ and $\e$, when they are small in the appropriate order.  
We will express our bound $M_{cri}$ in terms of the following parameters, in addition to the degree $n$ of the system:
\begin{itemize}
\item $M$: the number of connected components of the curve $\Dv=0$;
\item $N$: the number of contact points of the differential system.
\end{itemize}

\begin{theorem}\label{q}
Under the assumptions and notation of Theorem \ref{principall}, 
our differential system cannot have more than 
$1+(n-1)^2(M+N)$
limit cycles.
\end{theorem}

\begin{proof}
From Theorem \ref{recuentoo}, we must
compute the maximum number of limit behaviors which are contained in the
components of the curve $\Dv=0$.

Let $L\ge0$ be the number of connected components of $\Dv=0$ homeomorphic to a straight line.
Assume that the connected component $i$, for 
$i\in \{0, 1,\ldots,L\},$
contains $x_i\ge0$ contact
points ($x_0=0$ always). Then it can have at most $x_i+1$ limit  behaviors. Therefore
the number of limit critical closed paths contained in the components
of $\Dv=0$ homeomorphic to a straight line is at most 
$L+\sum_{i=0}^L x_i.$

Let $O\ge0$ be the number of connected components of $\Dv=0$ homeomorphic to an oval. 
Suppose $y_j\ge0$ is the number of contact points in the $j$-th component, for 
$j\in \{0, 1,\ldots,O\}$.
Note that $y_0=0$ always. Then we can have at most 
$\sum_{j=0}^O y_j$ 
different limit  behaviors, all of which are bounded.

In summary, the number of limit critical closed paths contained in
the components of $\Dv=0$ is at most
\[
L+\sum_{i=0}^{L} x_i+\sum_{j=0}^{O} y_j\le M+N.
\]
By Theorem \ref{multiplicidad}, each such possible limit behavior must
be multiplied by the corresponding multiplicity factor $(n-1)^2$. Hence, we
will have at most 
$1+(n-1)^2(M+N)$
critical closed paths of $E_\e$ with a value of $E_0$ uniformly away from zero, and for every $\e$
sufficiently small. We conclude by \eqref{cotapreliminar}.
\end{proof}

Note that in fact, the upper bound can be slightly improved to 
$1+(n-1)^2(L+N)$
where $L$ is the number of connected components of the curve $\Dv=0$ homeomorphic to a line. 

\subsection{Non-generic situation}\label{nongen}

In this final subsection we treat the case of polynomial differential
systems \eqref{e1} for which either the curve $\Dv=0$ has singular
points, i.e. the system
\begin{equation}\label{overdeteruno}
P_{xx}+Q_{yx}=P_{xy}+Q_{yy}=0,\quad P_x+Q_y=0
\end{equation}
has some solutions; our initial differential system \eqref{e1} has non-countable, infinitely many contact points,
i.e. system \eqref{salto0} has a continuum of solutions; or there are multiple solutions to the same system. Note that
equations \eqref{overdeteruno} and \eqref{salto0} involve the partial
derivatives of $\Dv$. 

Our argument revolves around the idea that vector fields $\bF$ for such systems can be uniformly approximated by a sequence $\bF_\delta$ of
non-singular polynomial vector fields without increasing the degree, and in
such a way that the divergence curve for $\bF_\delta$ has no
singularities, and finitely many simple contact points with system
\eqref{e1}.

We can definitely apply our previous results to the family of
functionals
\begin{align}
E_{\e, \delta}(\bu)=\int_0^1&\left[\frac12(\bF^\perp_\delta(\bu(t))\cdot\bu'(t))^2+\frac\e2(|\bu''(t)|^2+|\bu'(t)|^2+ |\bu(t)|^2)\right.\nonumber\\
&\left.+(\bu(t)\cdot\bv_{\delta}(t)+\bu'(t)\cdot\bv'_{\delta}(t)+\bu''(t)\cdot\bv''_{\delta}(t))\right.\nonumber\\
&\left.\frac1{2\e}(|\bv''_{\delta}(t)|^2+|\bv'_{\delta}(t)|^2+ |\bv_{\delta}(t)|^2)
\right]\,dt,\nonumber
\end{align}
where the dependence of the smooth paths $\bv_{\delta}$ on
$\delta$ is as regular as necessary. It is clear that the convergence
$E_{\e, \delta}\to E_\e\hbox{ as }\delta\to0$,
takes place in the Haussdorf sense since parameter $\delta$ is only involved in lower-order terms ($\e$ is fixed in this argument). 

Our first relevant observation is that our full discussion up to Subsection  \ref{ab} is valid regardless of whether our initial differential system \eqref{e1} is generic. This means that
$$
H(n)\le \lim_{a\to0}\lim_{\e\to0}\#(\{E_\e\le a\}\cap\bbO)
$$
where, as usual
$\#(\{E_\e\le a\}\cap\bbO)$
designates the number of connected components of the corresponding set. 
Because of the announced convergence $E_{\e, \delta}\to E_\e$, we
would also have the convergence
$$
\{E_{\e, \delta}\le a\}\to\{E_\e\le a\},\quad \delta\to0,
$$
in the Hausdorff distance for bounded sets. 
The operator $\#$, number of connected components of sub-level sets, is lower semicontinuous with respect to this convergence, and hence
$$
\#(\{E_\e\le a\}\cap\bbO)\le\lim_{\delta\to0}\#(\{E_{\e, \delta}\le a\}\cap\bbO).
$$
In this way, 
$$
H(n)\le \lim_{a\to0}\lim_{\e\to0}\#(\{E_\e\le a\}\cap\bbO)\le\lim_{a\to0}\lim_{\e\to0}\lim_{\delta\to0}\#(\{E_{\e, \delta}\le a\}\cap\bbO).
$$
If approximating polynomials $\bF_\delta$ are generic, all of our previous work implies that
$$
\lim_{a\to0}\lim_{\e\to0}\lim_{\delta\to0}\#(\{E_{\e, \delta}\le a\}\cap\bbO)\le C(n)
$$
where $C(n)$ is our upper bound for generic polynomial vector fields of degree $n$,
and hence we have the same bound in terms of the degree of the system for a non-generic differential vector field $\bF$. 

Notice that we are not claiming anything about the relationship
between limit cycles of $\bF$ and limit cycles of $\bF_\delta$, and their relative positions, or of
critical closed paths for $E_\e$ and of $E_{\e, \delta}$.

\section{Abstract results}\label{tres}
We gather in this section those results that can be shown in an abstract setting without specifying the nature of spaces or functionals. 
The fundamental fact that is driving us in this section is Theorem \ref{morsevalle} which we restate below (Theorem \ref{morsevalled}).

There is even a more general statement than Theorem \ref{morsevalle} that incorporates some fundamental topological invariants like the Euler characteristic, the Betti numbers, etc, and that can be stated in the context of infinite-dimensional manifolds (\cite{chang}, \cite{rothe1}-\cite{rothe}). Since all that we need for our goal in this article is inequality \eqref{morseinvalle} in Theorem \ref{morsevalled} below, we will restrict attention to the situation in the statement, and forget about those more general cases. 
Our plan is then to define and discuss the concepts involved in Definition \ref{mf} and all those mentioned in the statement of Theorem \ref{morsevalled} in preparation to prove those conditions for a perturbation $E_\e$ of a given initial functional $E_0$. 
Our discussion is intended right to the point with a minimal number of elements for a full proof.  

\subsection{Main concepts}\label{princon}
Unless otherwise explicitly stated, we will take, throughout this section, $E:\bbH\to\R$ to be a functional defined in an abstract Hilbert space $\bbH$ with elements $\bu$.
Recall (Definition \ref{mf}) that such a functional $E$
is called a Morse functional if:
\begin{itemize}
\item it is $\C^2$-, non-negative, coercive;
\item it enjoys the Palais-Smale property;
\item it has a finite number of critical points over each sub-level set
$\{E\le c\}$ for each non-critical value $c$, all of which are
non-degenerate and with a finite index.
\end{itemize}

The statement of Theorem \ref{morsevalled}, in addition to concepts in Definition \ref{mf},  involves
the notion of Morse index for a non-degenerate critical point of a
smooth functional defined on a Hilbert space. 

\begin{definition}
\begin{enumerate}
\item $E$ is coercive if
$$
E(\bu)\to\infty\hbox{ as }\|\bu\|\to\infty.
$$
\item $E$ is $\C^1$ if for every $\bu\in\bbH$ there is a linear functional 
$E'(\bu):\bbH\to\R$
such that
$$
\frac1{\|\bv\|}\|E(\bv)-\langle E'(\bu), \bv\rangle-E(\bu)\|\to0\hbox{ as }\|\bv\|\to0.
$$
The Riesz representation theorem implies that $E'(\bu)\in\bbH$.
\item An element $\bu\in\bbH$ is a critical point of a $\C^1$-functional $E$ if 
$E'(\bu)=\bcero\in\bbH$.
A real number
$c\in\R$ is a critical value of $E$ if there is a critical point
$\bu$, $E'(\bu)=\bcero$, such that $c=E(\bu)$.
\item $E$ is $\C^2$ if it is $\C^1$, and for each $\bu\in\H$, there is a bilinear, symmetric map 
$E''(\bu):\bbH\times\bbH\to\R$
such that
$$
\frac1{\|\bv\|^2}\|E(\bu)-E''(\bu)(\bv, \bv)-\langle E'(\bu), \bv\rangle-E(\bu)\|\to0\hbox{ as }\|\bv\|\to0.
$$
Again by the Riesz representation theorem, one can interpret the Hessian $E''(\bu)$ as a linear map
$$
E''(\bu):\H\to\H,\quad \langle E''(\bu)\bv, \bw\rangle=E''(\bu)(\bv, \bw).
$$
\end{enumerate}
\end{definition}
Suppose 
$E:\H\to\R$ 
is a non-negative, coercive,
$\C^2$-functional defined over a Hilbert space $\H$. Let $\bu\in\H$
be a critical point of $E$, i.e. $E'(\bu)=\bcero$. 

\begin{definition}
A critical point $\bu\in\H$ of $E$ is called {\it non-degenerate} is the self-adjoint operator 
$E''(\bu):\H\to\H$ 
is invertible. Otherwise, $\bu$ is said to be {\it degenerate}. When the number of
negative eigenvalues of $E''(\bu)$ is finite, such number is called
the {\rm (}Morse{\rm )} index of $\bu$.
\end{definition}

A main, indispensable property for Morse theory 
to hold is the {\it Palais-Smale condition}. It enables the passage of typical arguments in finite dimension to an infinite-dimensional setting.
For a
general, smooth $\C^1$- functional 
$E:\H\to\R,$ 
this important
compactness property reads:
\begin{quote}
If for a sequence $\{\bu_j\}$ we have that $E(\bu_j)\le K$ for all
$j$ and a fixed positive constant $K$, and $E'(\bu_j) \to\bcero$ as
$j\to\infty$, then a certain subsequence of $\{\bu_j\}$ converges
(strongly) in $\H$.
\end{quote}
If $E$ is coercive, we can replace the boundedness of $E$ along the
sequence $\{\bu_j\}$ by the uniform boundedness of $\{\bu_j\}$ in
$\H$. 

Though we have already talked about invariant sets of functionals in Theorem \ref{morsevalled}, for the sake of completeness we include here a more formal, precise definition. 
Suppose $E:\H\to\R$ is a $\C^2$-functional, and consider the steepest-descent flow $\bX$ of $-E$, i.e. 
$$
\bX'(t, \bu)=-E'(\bX(t, \bu))\hbox{ for }t>0,\quad \bX(0, \bu)=\bu.
$$
The regularity assumed on $E$ implies that the flow is defined for every positive $t$. For fixed positive $t$, the mapping 
$\bX(t, \cdot):\H\to\H$
can be easily shown to be $\C^1$. This is standard. 

\begin{definition}\label{valle}
A connected open subset $\bbS\subset\H$ is  said to be $E$-invariant, for a $\C^2$-functional $E$, if its boundary $\partial \bbS$ does not contain a critical point of $E$, and 
$\bX(t, \bbS)\subset\bbS$
for all positive $t$. The intersection of $E$-invariant sets is also $E$-invariant. When one deals with one functional, we will simply talk about invariant sets.
\end{definition}

Note that connected components of sublevel sets 
$$
\{E\le c\}=\{\bu\in\H: E(\bu)\le c\}
$$
for a non-critical value $c$ are invariant. If they are intersected with an additional invariant set, they remain invariant.

\subsection{Morse inequalities}\label{morin}
Morse inequalities in the full space $\H$ (Theorem \ref{morse}) is such a classical result that we will take it for granted without proof. 
In an infinite-dimensional setting, it can be found in several places, for instance,
Corollary (6.5.10) of \cite{berger} or Theorem 4.3 of Chapter 1 in
\cite{chang}. As a matter of fact, the version of Morse inequalities restricted to an invariant set $\bbI$ is not easy to find as such. Corollary (6.5.11) of \cite{berger} is a particular version of it when $\bbI$ is an invariant ball in $\H$. The proof of this case in that reference \cite{berger}, however, makes it very clear that the result is valid not just for a ball, but for a general invariant subset $\bbI$ which is topologically equivalent to a ball. 

\begin{theorem}\label{morsevalled}
Let $E:\H\to\R$ be a Morse functional (according to Definition \ref{mf}) over a Hilbert space $\H$. Let $\bbI$ be open, topologically equivalent to a ball, $E$-invariant, and put $M_k(\bbI)$ for the (finite) number of critical points of $E$ in $\bbI$, for each fixed index
$k$.  Then
\begin{gather}
M_0(\bbI)\ge1,\quad M_1(\bbI)-M_0(\bbI)\ge-1,\quad M_2(\bbI)-M_1(\bbI)+M_0(\bbI)\ge 1,\quad \dots,\nonumber\\
\sum_{k=0}^\infty (-1)^kM_k(\bbI)=1.\label{morseinvalle}
\end{gather}
\end{theorem}
\begin{proof}
If a set $\bbI$ is invariant for a Morse functional $E$, every concept or fact that depends exclusively on the flow of $-E$ can be restricted to $\bbI$ without change. 
On the other hand, since $\bbI$ is topologically equivalent to a ball, all of its homology groups coincide with those of the full space $\H$. Since Morse inequalities depend on those topological invariants, they remain valid when restricted to such invariant sets.
\end{proof}

When a certain base functional $E_0$ does not comply with the hypotheses of Theorem \ref{morsevalled}, there is a general, abstract procedure that permits to perturb it appropriately in such a way that the resulting perturbation $E_\e$ turns out to be a Morse functional, i.e. all of above requirements in Definition \ref{mf} and Theorem \ref{morsevalled} hold for $E_\e$. In addition, we would have to be convinced that the topology of sub-level sets $\{E_0\le a\}$ for $a$ near the absolute minimum of $E_0$ is not changed when we replace $E_0$ by $E_\e$.

\subsection{Perturbation}\label{pertur}

Suppose we have two nested Hilbert spaces $\H\subset\L$ with associated norms (coming from their respective inner products) $\|\cdot\|_\H$ and $\|\cdot\|_\L$, respectively. The norm $\|\cdot\|_\H$ is strictly finer (larger) than the restriction of $\|\cdot\|_\L$ to $\H$.
We take as a fact that bounded sequences in $\H$ are relatively compact in $\L$. The norm $\|\cdot\|$ always means $\|\cdot\|_\H$. 

Let $E_0:\L\to\R$ be a certain non-negative, $\C^2$-functional that turns out to be far from being a Morse functional so that Theorem \ref{morsevalled} cannot be applied to it. Our intention is to perturb it through an additional term of the form 
$$
\tilde E_\epsilon:\H\to\R, \quad \tilde E_\epsilon(\bu)=E_0(\bu)+\frac\epsilon2\|\bu\|_\H^2
$$
and show that a choice $\bv_0\in\H$ is possible in such a way that the modified functional 
$$
E_\epsilon:\H\to\R, \quad E_\epsilon(\bu)=E_0(\bu)+\frac\epsilon2\|\bu-\bv_0\|_\H^2
$$
is a Morse functional for every positive $\epsilon$. As a matter of fact, we have plenty of freedom to choose $\bv_0$ from. The success of this process requires the base functional $E_0$ to comply with the two compactness properties that follow:
\begin{enumerate}
\item $E'_0:\H\to\H$ (regarded as defined in $\H$) is a (non-linear) compact operator:
$$
\bu_j\rightharpoonup \bu\hbox{ implies }E'_0(\bu_j)\to E'_0(\bu).
$$
The sign $\rightharpoonup$ indicates weak convergence in $\H$, i.e.
$\langle\bu_j, \bv\rangle\to\langle\bu, \bv\rangle$
for every fixed $\bv\in\H$. 
\item For each $\bu\in\H$, the self-adjoint, linear operator
$E''_0(\bu):\H\to\H$
is compact.
\end{enumerate}

More specifically, we want to prove the two following results. Our starting point is 
the non-negative, $\C^2$-functional
$E_0:\H\subset\L\to\R$
where the injection $\H\hookrightarrow\L$ is compact. 

\begin{lemma}\label{importante}
Suppose $E'_0:\H\to\H$ is  a compact operator, and let $\V\subset\H$ be an open subset. There is $\bv_0\in\V$ 
such that the perturbed functional
\begin{equation}\label{perturbacion}
E_\e(\bu)=E_0(\bu)+\frac\e2\|\bu-\bv_0\|^2
\end{equation}
is non-negative, coercive, $\C^2$-, complies with the Palais-Smale
condition, and has a finite number (possibly depending on $\e$ and
$\a$) of non-degenerate  critical points in every finite sub-level
set of the form $\{E_\e<\a\}$ for arbitrary non-critical value $\a$. 
\end{lemma}

\begin{lemma}\label{eimportante}
Assume, in addition to the main assumption in Lemma \ref{importante}, that for each $\bu\in\H$, the linear, self-adjoint operator 
$E''_0(\bu):\H\to\H$
is compact. Let $E_\epsilon$ be the perturbation in that lemma. Then for each positive $\epsilon$, every critical point of $E_\epsilon$ is non-degenerate and has a finite Morse index. 
\end{lemma}

We can sum up these two fundamental facts into a single important, general statement.
\begin{theorem}\label{central}
Let $\H\hookrightarrow\L$ be two nested Hilbert spaces with compact injection. 
Let 
$E_0:\H\subset\L\to\R$
be a non-negative, $\C^2$-functional such that the derivative operators
$$
E'_0:\H\to\H,\quad E''_0(\bu):\H\to\H
$$ 
are compact, the second one 
for every fixed $\bu\in\H$. Then for every open subset $\V\subset\H$, there is 
$\bv_0\in\V$
such that the perturbed functional
\begin{equation}\label{perturbacion}
E_\e(\bu)=E_0(\bu)+\frac\e2\|\bu-\bv_0\|^2
\end{equation}
is a Morse functional.
\end{theorem}

For proving this main result, we need some preliminary abstract definitions
and facts, which we state next for the sake of completeness, most of which can be found in the classic book 
\cite{berger}, among others.

\subsection{More concepts}\label{mascon}
Suppose that 
$E:\H\to\R$ 
is a smooth $\C^2$-functional defined in a
Hilbert space $\H$. We shall use the following concepts in addition to the ones already introduced earlier. 
\begin{itemize}
\item[(i)]  An element $\bu\in\H$ is called a {\it regular} point for a non-linear 
$\C^1$-operator 
$\F:\H\to\H$ 
if the linear operator 
$\F'(\bu):\H\to\H$
is surjective. Otherwise, $\bu$ is called a {\it singular} point for $\F$. When $\F$ is the derivative of a $\C^2$-functional $E:\H\to\R$, then a critical point $\bu$ for $E$ is degenerate (respectively, non-degenerate) if it is singular (respectively, regular) for $\F=E'$. Note that in this case $\F'=E''$ is a self-adjoint operator, and so it is surjective if and only if it is bijective, see Section 2.7 in \cite{brezis}, for instance. The image $\F(\bu)$ of a singular point $\bu$ is called a singular value of $\F$.
\item[(ii)]  A mapping 
$\F:\H\to\H$ 
is a {\it non-linear Fredholm operator} if its Fr\'echet derivative 
$\F'(\bu):\H\to\H$
is a linear Fredholm map for each $\bu\in\H$. The {\it index} of $\F$ is defined to be the difference of the dimensions of the kernel and the cokernel of $\F'(\bu)$. This index is independent of $\bu$.
\item[(iii)]  The functional $E$ is a {\it Fredholm functional} if 
$E':\H\to\H$ 
is a Fredholm mapping, i.e. if 
$E''(\bu):\H\to\H$
is a linear Fredholm map for each $\bu\in\H$.
\end{itemize}

We state several interesting facts (page 100 in \cite{berger}).
\begin{proposition}\label{bas1}
The following statements hold.
\begin{itemize}
\item[a)] Any diffeomorphism between Banach spaces is a Fredholm map of index zero.
\item[b)] If $\F$ is a Fredholm map, and $\G$ is a compact operator, then the sum $\F+\G$ is also Fredholm with the same index as $\F$.
\end{itemize}
\end{proposition}

We recall two additional classic results. The first one is the
Inverse Function Theorem (page 113, \cite{berger}) for Banach spaces.
\begin{theorem}\label{bas2}
Let $\bbF$ be a $\C^1$-mapping defined in a neighborhood of some
point $\bu$ of a Banach space $\X$, with range in a Banach
space $\Y$. If $\bbF'(\bu)$ is a linear homeomorphism of
$\X$ onto $\Y$, then $\bbF$ is a local homeomorphism of a
neighborhood $\bU(\bu)$ of $\bu$ to a neighborhood
of $\bbF(\bu)$.
\end{theorem}
The second one is a version of Sard's theorem for
infinite-dimensional spaces (page 125 of \cite{berger}).
\begin{theorem}\label{bas3}
Let $\F$ be a $\C^q$-Fredholm mapping of a separable Banach space
$\X$ into a separable Banach space $\Y$. If $q>\max(\hbox{index
}\F, 0)$, the set of singular values of $\F$ are no-where dense (its
closure has empty interior) in $\Y$.
\end{theorem}

The proof of Theorem \ref{central}, through Lemmas \ref{importante} and \ref{eimportante}, will make use of Proposition
\ref{bas1}, and Theorems \ref{bas2} and \ref{bas3}, in addition to standard properties of compact operators. 

\subsection{Main proofs}\label{prueprin}
We are now in a position to prove  Lemma \ref{importante}. We begin by checking that the perturbation $E_\e$, for fixed $\e$, in
\eqref{perturbacion} complies with the Palais-Smale property, regardless of how the vector $\bv_0$ is selected : if $\{\bu_j\}$ is a sequence in $\H$ such that
$$
E_\e(\bu_j)\hbox{ is bounded},\quad E'_\e(\bu_j)\to\bc,
$$
then some subsequence of $\{\bu_j\}$ converges in $\H$.
Note that, since $E_\e$ is coercive in $\H$, we can replace the boundedness of $E_\e$ along the
sequence $\{\bu_j\}$ by the uniform boundedness of $\{\bu_j\}$ in
$\H$. 

Note first that each $E_\epsilon$ is coercive for fixed $\epsilon$. 
On the other hand, 
\begin{equation}\label{dercom}
E'_\e=E_0'+\e\id-\e\bv_0,
\end{equation}
where $\id:\H\to\H$ is the identity operator.

Suppose $\{\bu_j\}$ is uniformly bounded. Since $E'_0$ is assumed to be compact, there is a subsequence $\bu_j$ (not relabelled)
such that 
$$
E'_0(\bu_j)\to\overline\bu,\quad \overline\bu\in\H.
$$ 
To check the Palais-Smale conditions, if $E'_\varepsilon(\bu_j)\to\bcero$, due to
\eqref{dercom}, 
we have
$$
\varepsilon\bu_j=E'_\varepsilon(\bu_j)-E_0'(\bu_j)+\e\bv_0\to
-\overline\bu+\e\bv_0\hbox{ as }j\to\infty.
$$
Hence $\{\bu_j\}$ converges strongly in $\H$. This is exactly the
required property for each $E_\e$.

Consider now the functional 
$$
\tilde E_\e:\H\to\R,\quad \tilde E_\e(\bu)=E_0(\bu)+\frac\e2\|\bu\|^2.
$$
Its derivative 
$$
\bbF_\e(\bu)\equiv\tilde E'_\e(\bu)=\e\bu+E'_0(\bu)
$$ 
is the sum of a diffeomorphism, $\e\id$, and a compact operator, $E'_0$. By
Proposition \ref{bas1}, this derivative is a Fredholm operator of
index zero. By Theorem \ref{bas3}, the set of critical values of the
derivative $\bbF_\e$, that is
$$
\bbC_\e\equiv\{\bbF_\e(\bu)\in \H: \bbF'_\e(\bu)=\bc\},
$$
is no-where dense. The union 
$$
\bbC\equiv\cup_{\e>0}\bbC_\e
$$
is thus a meager or first-category subset. Notice here that we can, without loss of generality, restrict attention to some appropriate sequence of values for $\e$ and work henceforth with such a sequence. For the sake of simplicity, we will keep using $\e$ without indicating explicitly $\{\e_j\}$, and hope that this will not create confusion. After all, our main perturbation argument is valid if applied to some sequence of functionals $\{E_\e\}$ for some sequence of values for $\e$ converging to zero. 

Since $\H$ is a complete metric space, the classical Baire category theorem implies that $\bbC$ has empty interior, and consequently, we can choose an element
$\bv_0\in\V\setminus\bbC$,
with the properties claimed in the statement of
the theorem, so that every solution $\bu$ of the family of equations
$\bbF_\e(\bu)+\bv_0=\bc$
cannot be a singular point for none of the $\bbF'_\e$s, i.e. 
$\bbF'_\e(\bu)=\tilde
E''_0(\bu)$
is bijective, and so $\bu$ is non-degenerate. This
argument implies indeed that the  critical points of $E_\e$ are
non-degenerate, once $\bv_0$ has been chosen in this way and has
been added to $\tilde E_\e$, because 
$E''_\e(\bu)=\tilde
E''_\e(\bu)$.

The Inverse Function Theorem \ref{bas2} implies directly that non-degenerate
critical points of a $\C^2$-functional $E_\e$ are isolated. 

Finally, we argue why the number of critical points in sets of the form $\{ E_\e\le\alpha\}$ is finite. 
Indeed, if we 
let $\alpha$ be a positive real number and assume that there is
an infinite number $\{\bu_j\}$ of  critical points with
$$
E_\e(\bu_j)\le \alpha,\quad E'_\e(\bu_j)=\bc,
$$ 
the Palais-Smale condition for $E_\e$ would ensure the existence of a suitable
subsequence converging to some $\overline\bu$ which would be a
critical, non-isolated  point. This is a contradiction with the
previous statement about the fact that the critical points are isolated, and
so the number of such critical points has to be finite. This completes
the proof of Lemma \ref{importante}

The proof of Lemma \ref{eimportante} relies on the standard fact that eigenvalues of a linear,
self-adjoint, compact operator in a Banach space, like $E''_0(\bu)$,
always has a sequence of (real) eigenvalues converging to zero (see,
for instance, Chapter 6 of \cite{brezis}). By differentiating in 
\eqref{dercom}, 
$$
E''_\e(\bu)=E''_0(\bu)+\e\id,
$$
and hence eigenvalues of $E''_\e(\bu)$ are eigenvalues of
$E''_0(\bu)$ plus $\e$. The conclusion is that there cannot be an infinite number of
negative eigenvalues.

Theorem \ref{central} is then proved.

\subsection{Perturbation and critical points}\label{percri}
Let $E_0:\H\to\R$ be a $\cC^1$-, convex functional, bounded from below $E_0\ge M$ for some $M\in\R$, and let $E:\H\to\R^+$ be a $\cC^1$-, strictly convex functional. In most situations of interest, one would take
$$
E(\bu)=\frac12\|\bu\|^2\hbox{ or } E(\bu)=\frac12\|\bu-\bv_0\|^2
$$
for a fixed vector $\bv_0\in\H$. We will consider the perturbed functional
$$
E_\e:\H\to\R,\quad E_\e(\bu)=E_0(\bu)+\e E(\bu),
$$
which turns out to be strictly convex for every positive $\e$. As such, it admits a unique minimizer $\bu_\e$ which is determined as the unique solution of the critical point equation
$$
E'_0(\bu_\e)+\e E'(\bu_\e)=\bcero.
$$
This is standard. 

\begin{proposition}\label{convderivada}
In the situation just described, $E'_0(\bu_\e)\to\bcero$ in $\H$. 
\end{proposition}
\begin{proof}
Let $\bu\in\H$ be an arbitrary vector. Because $\bu_\e$ is a minimizer for $E_\e$, we have that
$$
E_0(\bu_\e)\le E_\e(\bu_\e)=E_0(\bu_\e)+\e E(\bu_\e)\le E_0(\bu)+\e E(\bu).
$$
If we take limits in $\e$, we conclude that
$$
\liminf_{\e\to0} E_0(\bu_\e)\le \limsup_{\e\to0} E_0(\bu_\e)\le E_0(\bu),
$$
and the arbitrariness of $\bu$ leads to
$$
\inf_\H E_0\le \liminf_{\e\to0} E_0(\bu_\e)\le \limsup_{\e\to0} E_0(\bu_\e)\le \inf_\H E_0.
$$
Hence $\{\bu_\e\}$ is minimizing for $E_0$, and because $E_o$ is bounded from below,  $E'_0(\bu_\e)\to\bcero$ (even if $\|\bu_\e\|\to\infty$).
\end{proof}

The convexity of the base functional $E_0$ looks unavoidable in this proof to ensure the uniqueness of minimizers $\bu_\e$, yet a stronger result is possible with no reference to convexity. Recall that statements concerning $\e\to0$ exactly means that it holds along some sequence of values of $\e$. This suffices for our purposes, as has been indicated earlier in this section. 
\begin{theorem}\label{ultimoz}
Let $E_0:\H\to\R$ be a non-negative, $\cC^1$-functional, and let $\{\bu_\e\}$ be a branch of critical points for the family of functionals
$$
E_\e(\bu)=E_0(\bu)+\frac{\e}2\|\bu-\bv_0\|^2
$$
for a fixed element $\bv_0\in\H$, i.e.
\begin{equation}\label{eccri}
E_0'(\bu_\e)+\e(\bu_\e-\bv_0)=\bcero
\end{equation}
for every positive $\e$ sufficiently small. Then: 
\begin{enumerate}
\item $E'_0(\bu_\e)\to\bcero$ in $\H$ as $\e\to0$; and
\item there is no such branch with
\begin{equation}\label{condiciones}
E_0(\bu_\e)\to0,\quad E_\e(\bu_\e)>a,
\end{equation}
for a fixed value $a>0$.
\end{enumerate}
\end{theorem}
\begin{proof}
The first observation is that the branch $\{\bu_\e\}$ is differentiable with respect to $\e$ as a consequence of the differentiability of $E_0$ and the smooth dependence of solutions of critical equations with respect to parameters. Let $\delta>0$ be fixed. Then
$$
E_\delta(\bu_\delta)-\lim_{\e\to0}E_\e(\bu_\e)=\int_0^\delta \frac d{d\e}[E_\e(\bu_\e)]\,d\e.
$$
The differentiation (with respect to $\e$) under the integral sign together with \eqref{eccri} leads immediately to
$$
E_\delta(\bu_\delta)-\lim_{\e\to0}E_\e(\bu_\e)=\int_0^\delta \frac12\|\bu_\e-\bv_0\|^2\,d\e.
$$
We will express this identity in the form
\begin{equation}\label{importantezz}
\lim_{\e\to0}\left[E_0(\bu_\e)+\frac\e2\|\bu_\e-\bv_0\|^2\right]+\int_0^\delta\frac12\|\bu_\e-\bv_0\|^2\,d\e=E_0(\bu_\delta)+\frac\delta2\|\bu_\delta-\bv_0\|^2.
\end{equation}
From this basic equality we will conclude the two claimed facts.

Since all terms involved in \eqref{importantezz} are non-negative, we see that
$$
\lim_{\e\to0}\frac\e2\|\bu_\e-\bv_0\|^2\le E_\delta(\bu_\delta).
$$
On the other hand, \eqref{eccri} implies that both vectors $E'_0(\bu_\e)$ and $\bu_\e-\bv_0$ are co-linear and, in addition,
$$
\lim_{\e\to0}-\frac12 E'_0(\bu_\e)\cdot(\bu_\e-\bv_0)=\lim_{\e\to0}\frac\e2\|\bu_\e-\bv_0\|^2\le E_\delta(\bu_\delta).
$$
Since the upper bound on the right-hand side is independent of $\e$, if $\|\bu_\e-\bv_0\|$ converges to infinity, then $E'_0(\bu_\e)$ must converge necessarily to zero. On the other hand, if $\bu_\e-\bv_0$ is uniformly bounded, then \eqref{eccri} clearly implies that $E'_0(\bu_\e)$ converges to zero as well. At any rate, $E'_0(\bu_\e)\to\bcero$ in $\bbH$. 

Concerning our second point, suppose, seeking a contradiction, that we could find a branch of critical paths $\{\bu_\e\}$, complying with \eqref{eccri}, for which \eqref{condiciones} holds for some fixed $a>0$. Then
$$
0<a\le \lim_{\e\to0}E_\e(\bu_\e)=\lim_{\e\to0}\left[E_0(\bu_\e)+\frac\e2\|\bu_\e-\bv_0\|^2\right]=\lim_{\e\to0}\frac\e2\|\bu_\e-\bv_0\|^2,
$$
and \eqref{importantezz} leads to
$$
a+\int_0^\delta\frac12\|\bu_\e-\bv_0\|^2\,d\e\le E_0(\bu_\delta)+\frac\delta2\|\bu_\delta-\bv_0\|^2.
$$
If $\|\bu_\e-\bv_0\|^2$ tends to infinity (as $\e\to0$), then for some selected sequence of values for $\delta$ tending to zero, we should have
$$
\frac{\delta}2\|\bu_\delta-\bv_0\|^2\le \int_0^\delta\frac12\|\bu_\e-\bv_0\|^2\,d\e.
$$
Taking this fact into the previous inequality, we realize that
$$
a+\frac{\delta}2\|\bu_\delta-\bv_0\|^2\le E_0(\bu_\delta)+\frac\delta2\|\bu_\delta-\bv_0\|^2,
$$
and $a\le E_0(\bu_\delta)$ for some sequence of values $\delta$ tending to zero. This is a contradiction with \eqref{condiciones}. If, on the other hand, $\|\bu_\e-\bv_0\|^2$ is uniformly bounded, then  as before
$$
0<a\le\lim_{\e\to0} E_\e(\bu_\e)=\lim_{\e\to0} E_0(\bu_\e)=0,
$$
which is again a contradiction. 
\end{proof}

\section{General results}\label{cuatroz}

\subsection{Critical paths of functionals}

The object of this subsection is to derive and study the differential equations
which must satisfy critical closed paths of local, integral functionals (like $E_\e$) in
$\espacioo$. The proof uses standard ideas in the Calculus of Variations,
but we include them because their understanding is crucial for our counting procedure. 

\subsubsection{Second-order problems}\label{aqui}
Let 
$$
F(t, \mathbf{u}, \bz, \bZ):[0,
1]\times\R^2\times\R^2\times\R^2\to\R
$$ 
be a $(\C^\infty$-) function
with respect to $(t, \bu, \bz, \bZ)$, with partial derivatives
$F_\bu$, $F_\bz$, $F_\bZ$. Assume that 
$F_{\mathbf{u}}(t, \mathbf{v}, \mathbf{v}',
\bv'')$
and
$$
F_{\mathbf{z}}(t, \mathbf{v}, \mathbf{v}',
\bv'')-\int_0^tF_{\mathbf{u}}(t, \mathbf{v}, \mathbf{v}', \bv'')\,ds
$$
belong to $L^1((0, 1); \R^2)$ for every feasible
$\mathbf{v}\in\espacioo$.

\begin{theorem}\label{basicsecond}
Suppose that the functional
$$
E(\bu)=\int_0^1 F(t, \mathbf{u}(t), \mathbf{u}'(t),
\mathbf{u}''(t))\,dt
$$
admits a  critical closed path 
$$
\mathbf{u}:[0, 1]\to\R^2\hbox{ in }\espacioo.
$$ 
Then the function
\begin{equation}\label{este}
\frac d{dt}F_\bZ(t, \mathbf{u}, \mathbf{u}', \bu'')-F_\bz(t,
\mathbf{u}, \mathbf{u}', \bu'')
\end{equation}
is absolutely continuous in $[0, 1]$, and
\begin{equation}\label{aaa}
\frac d{dt}\left(\frac{d}{dt}F_\bZ(t, \bu, \bu', \bu'')-F_\bz(t, \bu,
\bu', \bu'')\right)+F_\bu(t, \bu, \bu', \bu'')=\bcero \,\,\,\hbox{for a.e.
$t$ in }(0, 1).
\end{equation}
Moreover
\begin{equation}\label{saltoooe}
[F_\bZ(t, \bu, \bu', \bu'')]_{t=0}=\bcero.
\end{equation}
\end{theorem}

Brackets in \eqref{saltoooe} indicate the jump of the field inside at
the time indicated (difference between $t=1$, and $t=0$), that is
\begin{align}
[F_\bZ(t, \bu, \bu', \bu'')]_{t=0}=&\left.F_\bZ(t, \bu(t), \bu'(t),
\bu''(t))\right|_{t=1^-}\nonumber\\
&-\left.F_\bZ(t, \bu(t), \bu'(t),
\bu''(t))\right|_{t=0^+}.\nonumber
\end{align}

Notice that the integrability demanded on those combinations of
partial derivatives of $F$ in the statement of Theorem
\ref{basicsecond}  is equivalent to having 
$$
F_{\mathbf{u}}(t,
\mathbf{v}, \mathbf{v}', \bv'')\hbox{ and }F_{\mathbf{z}}(t, \mathbf{v},
\mathbf{v}', \bv'')
$$ 
integrable for every feasible
$\mathbf{v}\in\espacioo$.
We have however decided to keep the
statement as it is for that is exactly the form in which those
combinations of partial derivatives will occur in the proof. 

\begin{proof}[Proof of Theorem \ref{basicsecond}]
Take 
$\bU\in\espacioo$.
If $\bu$ is a  critical closed path of $E$,
then
$$
\left.\frac d{d\delta} E(\bu+\delta\bU)\right|_{\delta=0}=0,
$$
that is to say
$$
\left.\frac d{d\delta}\right|_{\delta=0}\int_0^1 F(t,
\bu(t)+\delta\bU(t), \bu'(t)+\delta\bU'(t),
\bu''(t)+\delta\bU''(t))\,dt=0.
$$
This derivative has the form
\begin{equation}\label{eulerlagrangee}
\int_0^1 \left[F_\bu(t, \bu, \bu', \bu'')\cdot \bU(t)+ F_\bz(t, \bu,
\bu', \bu'')\cdot \bU'(t)+F_\bZ(t, \bu, \bu', \bu'')\cdot
\bU''(t)\right]\,dt=0.
\end{equation}
We consider the special subspace $\bbU$ of variations $\bU$ defined
by
\begin{equation}\label{ortho}
\bbU=\{\bU\in\espacioo: \bU(0)=\bU(1)=\bcero, \bU''\in\{1, t\}^\perp\},
\end{equation}
where $\{1, t\}^\perp$ is the orthogonal complement, in $L^2([0, 1];
\R^2)$, of the subspace generated by $\{1, t\}$. Since these
orthogonality conditions mean
$$
\int_0^1\bU''(t)\,dt=\bU'(1)-\bU'(0)=\bcero,\quad \int_0^1
t\bU''(t)\,dt=\bU'(1)=\bcero,
$$
we can also put
$$
\bbU=\{\bU\in\espacioo: \bU(0)=\bU(1)=\bU'(0)=\bU'(1)=\bcero\}.
$$
We also set
\begin{gather}
\Psi(t)=\int_0^t F_\bu(s, \bu(s), \bu'(s), \bu''(s))\,ds,\nonumber\\
\Phi(t)=\int_0^t \left[-\Psi(s)+F_\bz(s, \bu(s), \bu'(s),
\bu''(s))\right]\,ds,\nonumber
\end{gather}
two continuous, bounded functions by hypothesis. For $\bU\in\bbU$, an
integration by parts in the first term of (\ref{eulerlagrangee})
yields
$$
\int_0^1\left[-\Psi(t)\cdot \bU'(t)+F_\bz(t, \bu, \bu', \bu'')\cdot
\bU'(t)+F_\bZ(t, \bu, \bu', \bu'')\cdot \bU''(t)\right]\,dt=0,
$$
because 
$$
\left.\Psi(t)\bU(t)\right|_0^1=\bcero
$$
for test fields
$\bU\in\bbU$. A second integration by parts leads to
$$
\int_0^1\left[-\Phi(t)+F_\bZ(t, \bu, \bu',
\bu'')\right]\cdot\bU''(t)\,dt=0,
$$
again because 
$$
\left.\Phi(t)\bU'(t)\right|_0^1=\bcero
$$ 
if $\bU\in\bbU$.
Due to the arbitrariness of $\bU\in\bbU$, according to \eqref{ortho}
we conclude that
\begin{equation}\label{conclusione}
F_\bZ(t, \bu, \bu', \bu'')-\Phi(t)=c+Ct\hbox{ in }(0, 1),
\end{equation}
with $c$ and $C$ constants. In particular, since $\Phi$ is absolutely
continuous (it belongs to $W^{1, 1}((0, 1);$ $ \R^2)$), we know that
$F_\bZ(t, \bu, \bu', \bu'')$ must  be absolutely continuous too in
$(0, 1)$, and as such, it cannot have jumps in $(0, 1)$, though it
could possibly have at the endpoints. By differentiating once in
(\ref{conclusione}) with respect to $t$,
$$
\frac{d}{dt}F_\bZ(t, \bu, \bu', \bu'')-F_\bz(t, \bu, \bu',
\bu'')+\Psi(t)=C \hbox{ a.e. in }(0, 1),
$$
and even further
\begin{equation}\label{ultima}
\frac d{dt}\left(\frac{d}{dt}F_\bZ(t, \bu, \bu', \bu'')-F_\bz(t, \bu,
\bu', \bu'')\right)+F_\bu(t, \bu, \bu', \bu'')=0\hbox{ a.e. in }(0,
1).
\end{equation}
We take this information back to \eqref{eulerlagrangee} for a general
$\bU\in\espacioo$,
not necessarily belonging to $\bbU$. One
integration by parts in the second term in (\ref{eulerlagrangee})
yields
\begin{gather}
\int_0^1 F_\bz(t, \bu, \bu', \bu'')\cdot \bU'(t)\,dt=-\int_0^1\frac
d{dt} F_\bz(t, \bu, \bu', \bu'')\cdot\bU(t)\,dt\nonumber\\
+ [F_\bz(t, \bu, \bu',
\bu'')]_{t=0}\cdot \bU(0).\nonumber
\end{gather}
Recall the periodicity conditions for $\bU$. Two such integrations by
parts in the third term of \eqref{eulerlagrangee} leads to
\begin{align}
\int_0^1 F_\bZ(t, \bu, \bu',\bu'')\cdot \bU''(t)\,dt=&-\int_0^1\frac d{dt}F_\bZ(t, \bu, \bu',\bu'')\cdot\bU'(t)\,dt\nonumber\\
&+[F_\bZ(t, \bu, \bu', \bu'')]_{t=0}\cdot \bU'(0)\nonumber\\
=&\int_0^1\frac{d^2}{dt^2}F_\bZ(t, \bu, \bu',\bu'')\cdot \bU(t)\,dt\nonumber\\
&-[\frac d{dt}F_\bZ(t, \bu, \bu', \bu'')]_{t=0}\cdot \bU(0)\nonumber\\
&+[F_\bZ(t,
\bu, \bu', \bu'')]_{t=0}\cdot \bU'(0).\nonumber
\end{align}
In this way (\ref{eulerlagrangee}) becomes
\begin{gather}
\int_0^1\left[\frac d{dt}\left(\frac{d}{dt}F_\bZ(t, \bu, \bu',
\bu'')-F_\bz(t, \bu,
\bu', \bu'')\right)+F_\bu(t, \bu, \bu', \bu'')\right]\cdot \bU(t)\,dt\nonumber\\
-[\frac d{dt}F_\bZ(t, \bu, \bu', \bu'')-F_\bz(t, \bu, \bu',
\bu'')]_{t=0}\cdot \bU(0)+[F_\bZ(t, \bu, \bu', \bu'')]_{t=0}\cdot
\bU'(0).\nonumber
\end{gather}
The integral here vanishes precisely by \eqref{ultima}, and so we are
only left with the contributions on the end-points. Hence, we obtain
\begin{equation}\label{extremoss}
[\frac d{dt}F_\bZ(t, \bu, \bu', \bu'')-F_\bz(t, \bu, \bu',
\bu'')]_{t=0}\cdot \bU(0)-[F_\bZ(t, \bu, \bu', \bu'')]_{t=0}\cdot
\bU'(0)=0.
\end{equation}
Since vectors $\bU'(0)$ and $\bU(0)$ can be chosen arbitrarily, and
independently of each other, because there is always a path 
$\bU\in\espacioo$
starting in a certain arbitrary vector of $\R^2$ and with
any preassigned velocity,  we conclude that
\begin{equation}\label{jump1}
[F_\bZ(t, \bu, \bu', \bu'')]_{t=0}=\bcero,
\end{equation}
and
\begin{equation}\label{jump2}
[\frac d{dt}F_\bZ(t, \bu, \bu', \bu'')-F_\bz(t, \bu, \bu',
\bu'')]_{t=0}=\bcero.
\end{equation}
This completes the proof of Theorem \ref{basicsecond}. Note that this
last condition implies that
$$
\frac d{dt}F_\bZ(t, \bu, \bu', \bu'')-F_\bz(t, \bu, \bu', \bu'')
$$
is absolutely continuous in the interval $[0, 1]$, including the
endpoints.
\end{proof}

\subsection{Initial-value, Cauchy problems, and the shooting method}
Solutions whose existence is guaranteed by Theorem \ref{basicsecond} are not, in general, unique. To enforce such uniqueness, which will be a necessary ingredient of our counting procedure, we need to resort to standard initial-value Cauchy problems, and relate them to periodicity. The following proposition is very classical.
\begin{proposition}\label{hesta}
Consider the initial-value, Cauchy problem associated with the fourth-order differential system \eqref{aaa}, under the same hypotheses as in Theorem \ref{basicsecond}, 
\begin{equation}\label{segundoorden}
\frac d{dt}\left(\frac{d}{dt}F_\bZ(t, \bu, \bu', \bu'')-F_\bz(t, \bu,
\bu', \bu'')\right)+F_\bu(t, \bu, \bu', \bu'')=\bcero\hbox{ for }t\in(0, 1),
\end{equation}
under the initial conditions
$$
\bu^{i)}(0; \bp_0, \bp_1, \bp_2, \bp_3)=\bp_i,\quad i=0, 1, 2, 3.
$$
for arbitrary $\bp_i\in\R^2$. There is a unique solution map
$$
\bu(t; \bp_0, \bp_1, \bp_2, \bp_3):[0, 1]\times\R^2\times\R^2\times\R^2\times\R^2\to\R^2
$$
which inherits the same smoothness as that of $F$ in of all its variables. In particular, if $F$ is analytic in all its variables, so is $\bu$. 
\end{proposition}

For future reference, we formally adopt the following notation.
\begin{definition}\label{solutionmap}
The mapping
$$
\bu(t; \bp_0, \bp_1, \bp_2, \bp_3):[0, 1]\times\R^2\times\R^2\times\R^2\times\R^2\to\R^2
$$
will designate the solution mapping for problem \eqref{segundoorden}. $\bu$ is analytic in all its variables if so is the integrand $F$.
\end{definition}

We are interested in estimating how many four-tuples $(\bp_0, \bp_1, \bp_2, \bp_3)$ are capable of producing $1$-periodic solutions through the solution mapping of the corresponding initial-value problem, i.e. such that
$$
\bu^{i)}(0; \bp_0, \bp_1, \bp_2, \bp_3)=\bu^{i)}(1; \bp_0, \bp_1, \bp_2, \bp_3)
$$
for both $i=0, 1$. To this end, we invoke the shooting method. 

The shooting method is a mechanism that aims at solving an end-point-value problem through a typical initial-value, Cauchy one. 
What is important for us, as stressed in the preceding sections, is the fundamental distinction between an initial-value  and a periodic problem in terms of uniqueness of solutions: under very reasonable regularity hypotheses for the differential problem, the solution of an initial-value problem is unique; however, this may not be so for a periodic version of it. We are very much in need of controlling the possible number of solutions of an end-point-value problem. We plan to do it through the corresponding initial-value version of the differential problem, as we explain next.

\begin{corollary}\label{importantez}
Let 
$$
\bu(t; \bp_0, \bp_1, \bp_2, \bp_3):[0, 1]\times\R^2\times\R^2\times\R^2\times\R^2\to\R^2
$$
be the analytic solution mapping in Definition \ref{solutionmap} and Proposition \ref{hesta} for an integrand 
$$
F(t, \mathbf{u}, \bz, \bZ):[0,
1]\times\R^2\times\R^2\times\R^2\to\R
$$ 
in Theorem \ref{basicsecond} which is analytic. 
Except for an isolated set of pairs $(\bp_0, \bq_0)$, there is a well-defined, analytic mapping
$$
\overline\bu(t; \bp, \bq): [0, 1]\times{\mathbb B}_\rho(\bp_0)\times{\mathbb B}_\rho(\bq_0)\to\R^2
$$
for respective neighborhoods ${\mathbb B}_\rho(\bp_0)$ and ${\mathbb B}_\rho(\bq_0)$ of $\bp_0$ and $\bq_0$, furnishing the solution of \eqref{segundoorden} with
$$
\bu(0; \bp, \bq)=\bu(1; \bp, \bq)=\bp,\quad \bu'(0; \bp, \bq)=\bu'(1; \bp, \bq)=\bq
$$
for 
$$
(\bp, \bq)\in{\mathbb B}_\rho(\bp_0)\times{\mathbb B}_\rho(\bq_0).
$$
\end{corollary}
\begin{proof}
Consider the two conditions
\begin{equation}\label{implicita}
\bu(1; \bp, \bq, \bP, \bQ)-\bp=\bcero,\quad \bu'(1; \bp, \bq, \bP, \bQ)-\bq=\bcero.
\end{equation}
By the Implicit Function Theorem, these identities determine $(\bP, \bQ)$, as analytic functions of $(\bp, \bq)$, provided the jacobian of the two functions in the left-hand sides of \eqref{implicita} with respect to $(\bP, \bQ)$ does not vanish. Since all functions involved in such jacobian are analytic, that vanishing condition can only occur in an isolated number of pairs. If, whenever defined, we put for such implicit analytic mapping
$(\bP, \bQ)=\bbF(\bp, \bq)$,
it suffices to take
$\overline\bu(t; \bp, \bq)=\bu(t; \bp, \bq, \bbF(\bp, \bq))$.
\end{proof}

We will invoke Corollary \ref{importantez} for the family of integrands 
$$
F_\e(t, \bu, \bz, \bZ)=\frac12(\bz\cdot\bF^\perp(\bu))^2+\frac\e2\left(|\bZ-\bv''_0(t)|^2+|\bz-\bv'_0(t)|^2+|\bu-\bv_0(t)|^2\right),
$$
where 
$$
\bF^\perp(\bu)= (-Q(\bu),P(\bu)),\quad \bv_0=(X, Y).
$$
For each fixed, positive $\e$, $F_\e$ is analytic in all its variables provided the auxiliary path $\bv_0(t)$ can be taken to be analytic as well. There is no problem with such a selection as has been described in Section \ref{tres}. 

\subsection{Multiplicity}\label{multiplicidadd}
Finally, we discuss some simple results to deal with multiple solutions of the critical point equation of the last subsection. 

Consider the family of functions
\begin{equation}\label{funcionesss}
f_j(r): I\to\R,\quad f_j(r)=\int_J \rho_j(t) P(u(t, r))\,dt, \quad j=1, 2, \dots,
\end{equation}
where both $I$ and $J$ are closed, finite intervals, $0\in J$, $|J|=1$, $u(t, r)$ is a certain given function, and $P(x)$ is a polynomial of degree $n$ of a single variable $x$. If the sequence of functions $\{\rho_j\}$ is a sequence of mollifiers, or an approximation of the identity, in the sense
\begin{equation}\label{mollifierr}
\lim_{j\to\infty}\int_J \rho_j(t) F(t)\,dt=F(0)
\end{equation}
for every continuous $F$, then we have the following interesting result. Recall that the main properties of the sequence $\{\rho_j\}$ (in the one-dimensional case) are:
\begin{itemize}
\item smoothness: $\rho_j\in \cC^\infty(J)$;
\item small support: 
$$
J_j\equiv\supp(w_j)= [-\alpha_j, \alpha_j)]\subset J, \alpha_j\searrow0;
$$
\item non-negativeness and unit total mass:  
$$
\int_J \rho_j(t)\,dt=1,\quad  \rho_j\ge0.
$$
\end{itemize}
The convergence properties associated with such a family of functions go well beyond \eqref{mollifierr}.
This is very classical material (\cite{brezis}).

\begin{proposition}\label{mollifier}
Assume that the function 
$u(t, r):J\times I\to\R$
is smooth, and $u(0, r)$ is linear in $r$. Then, for arbitrarily large $j$, the family of functions in \eqref{funcionesss} 
$$
f_j(r): I\to\R,\quad f_j(r)=\int_J \rho_j(t) P(u(t, r))\,dt, \quad j=1, 2, \dots,
$$
for a polynomial $P(x)$ of degree $n$, cannot have more than $n$ roots in $I$.
\end{proposition}
\begin{proof}
Suppose for every large $j$, there are, at least, $n+1$ roots (counting multiplicity) in $I$ for $f_j$. Then it is elementary to realize that there must be some $r_j\in I$ which is a root of the $n$-th derivative
$$
0=f_j^{n)}(r_j)=\int_J \rho_j(t)\left.\frac{\partial^n}{\partial r^n}\right|_{r=r_j}P(u(t, r))\,dt.
$$
Differentiation under the integral sign is legitimate given the smoothness of all elements in a compact set. 
Let $r_j\to r_\infty\in I$ for some subsequence not relabeled. Then
\begin{equation}\label{convergencia}
0=\lim_{j\to\infty} \int_J \rho_j(t)\left.\frac{\partial^n}{\partial r^n}\right|_{r=r_j}P(u(t, r))\,dt=\frac{\partial^n}{\partial r^n}P(u(0, r_\infty)).
\end{equation}
This is impossible for a non-trivial polynomial of degree $n$ under the linearity of $u(0, r)$ on $r$. 

Note that the convergence expressed in \eqref{convergencia} is a consequence of writing the difference 
$$
\int_J \rho_j(t)\left[\frac{\partial^n}{\partial r^n}P(u(t, r_j))-\frac{\partial^n}{\partial r^n}P(u(0, r_\infty))\right]\,dt
$$
in the form
\begin{gather}
\int_J \rho_j(t)\left[\frac{\partial^n}{\partial r^n}P(u(t, r_j))-\frac{\partial^n}{\partial r^n}P(u(t, r_\infty))\right]\,dt+\nonumber\\
\int_J \rho_j(t)\left[\frac{\partial^n}{\partial r^n}P(u(t, r_\infty))-\frac{\partial^n}{\partial r^n}P(u(0, r_\infty))\right]\,dt.\nonumber
\end{gather}
Take into account the smoothness of $u(t, r)$ and the boundedness of its derivatives up to order $n$ in $J\times I$, and the fact that
$$
\int_J\rho_j(t)\,dt=1,\hbox{ for all }j,
$$
to use the Lebesgue dominated convergence principle for the first integral. 
This is standard.
\end{proof}

The same proof es valid for the family of functions
\begin{equation}\label{funcioness}
f_j(r): I\to\R,\quad f_j(r)=\int_J \rho_j(t) u(t, r)\,dt, \quad j=1, 2, \dots,
\end{equation}
where other ingredients are kept exactly the same. 
\begin{proposition}\label{mollifierdos}
Assume that the function 
$u(t, r):J\times I\to\R$
is smooth, and $u(0, r)$ is a non-trivial polynomial of degree $n$ in the variable $r$. Then, for arbitrarily large $j$, the family of functions in \eqref{funcioness} 
$$
f_j(r): I\to\R,\quad f_j(r)=\int_J \rho_j(t) u(t, r)\,dt, \quad j=1, 2, \dots,
$$
cannot have more than $n$ roots in $I$.
\end{proposition}

We can even let the function $u(t, \br)$ be the result of an asymptotic process depending on several variables $(t, \br)$
$$
f_j(\br): I\to\R,\quad f_j(\br)=\int_J \rho_j(t) u_j(t, \br)\,dt, \quad j=1, 2, \dots,
$$
provided we are allowed to modulate the mollifier $\rho_j$ with respect to $u_j$. Recall that 
$$
J_j=[-\alpha_j, \alpha_j],\quad \alpha_j\searrow0.
$$

\begin{proposition}\label{aclarar}
Let $\{u_j(t, \br)\}$ be a sequence of smooth functions of several variables 
$(t, \br)\in J\times I\subset\R\times\R^d$.
\begin{enumerate}
\item For every subsequence $k(j)$ sufficiently advanced, 
$$
\lim_{j\to\infty}\int_{J_{k(j)}}\rho_{k(j)}(t)u_j(t, \br)\,dt=\lim_{j\to\infty}u_j(0, \br)
$$
for all $\br$ in a compact set of $I$.
\item For every subsequence $k(j)$ sufficiently advanced, 
$$
\lim_{j\to\infty}\int_0^{2\alpha_{k(j)}}\rho_{k(j)}(t)u_j(t, \br)\,dt=\lim_{j\to\infty}u_j(\alpha_{k(j)}, \br)
$$
for all $\br$ in a compact set of $I$. 
\end{enumerate}
\end{proposition}
\begin{proof}
The proof is straightforward after the ideas utilized in the proof of Proposition \ref{mollifier}. For each $j$ fixed, we have
$$
\lim_{k\to\infty}\int_J \rho_k(t) u_j(t, \br)\,dt=u_j(0, \br).
$$
For $\br$ belonging to a given compact set, we can select $k(j)$ sufficiently advanced (depending on $j$ and such compact set), so that 
$$
\left|\int_J \rho_{k(j)}(t) u_j(t, \br)\,dt-u_j(0, \br)\right|\le\frac1j.
$$
The second part is immediate under an easy change of variables, and the fact
$$
\lim_{k\to\infty}\int_{J_k}\rho_k(t)\left[u_j(s+\alpha_k, \br)-u_j(\alpha_k, \br)\right]\,dt=0,
$$
for every fixed $j$, which is a consequence of the first statement. 
\end{proof}

\section{Specific results}\label{concretos}
We show in this final section the compactness properties of the derivatives of $E_0$ required to apply abstract results in Section \ref{tres} to our particular framework.

For our functional $E_0$, it is easy to find an expression for
$$
\langle E'_0(\bu), \bU\rangle,\quad \bu, \bU\in \espacioo.
$$
Indeed, by definition we have
\begin{equation}\label{derivada}
\langle E'_0(\bu), \bU\rangle=\left.\frac d{d\tau}\right|_{\tau=0}
E_0(\bu+\tau\bU).
\end{equation}
Since
$$
E_0(\bu+\tau\bU)=\frac12\int_0^1[\bF^\perp(\bu+\tau\bU)\cdot(\bu'+\tau\bU')]^2\,dt,
$$
then from \eqref{derivada} we have
\begin{equation}\label{prider}
\langle E'_0(\bu), \bU\rangle=\int_0^1(\bF^\perp(\bu)\cdot\bu')
[(D\bF^\perp(\bu)\bU)\cdot \bu'+\bF^\perp(\bu)\cdot\bU']\,dt.
\end{equation}
We are, therefore, seeking an element 
$$
\bv(=E'_0(\bu))\in\espacioo
$$
such that
$\bv\cdot\bU=\langle E'_0(\bu), \bU\rangle$,
 that is to say
\begin{equation}\label{formadebil0}
\int_0^1(\bv\cdot\bU+\bv'\cdot\bU'+\bv''\cdot\bU'')\,dt=\int_0^1(\bF^\perp(\bu)\cdot\bu') [(D\bF^\perp(\bu)\bU)\cdot
\bu'+\bF^\perp(\bu)\cdot\bU']\,dt
\end{equation}
for every 
$\bU\in\espacioo$.
There is a unique such $\bv$, which
turns out to be the minimizer (with respect to 
$\bU\in\espacioo)$
of the augmented functional
\begin{equation}\label{cuad}
\int_0^1[\frac12|\bU''|^2+\frac12|\bU'|^2+\frac12|\bU|^2-
(\bF^\perp(\bu)\cdot\bu') [(D\bF^\perp(\bu) \bU)\cdot\bu'+
\bF^\perp(\bu)\cdot\bU']]\,dt.
\end{equation}
The existence of a unique minimizer for this problem, which is
quadratic, is a direct consequence of the classical Lax-Milgram Theorem (see
Corollary 5.8 of \cite{brezis} for instance).

Therefore the equation for 
$$
\bv=E'_0(\bu)\in\espacioo
$$ 
will be the
associated Euler-Lagrange system for the functional \eqref{cuad} as
it is given by this last theorem
\begin{equation}\label{derivadas}
[\bv'']''-[\bv'+(\bF^\perp(\bu)\cdot\bu')\bF^\perp(\bu)]'+\bv+
(\bF^\perp(\bu)\cdot\bu')\bu'^T D\bF^\perp(\bu)=\bc\hbox{ in }(0, 1).
\end{equation}
Its weak formulation is exactly \eqref{formadebil0}.

Let us next deal with the second derivative. 
If, in general, we have a certain smooth, $\C^2$-functional
$E:\H\to\R$ over a Hilbert space $\H$ with derivative
$E':\H\to\H$, there are various ways to deal with the second
derivative, but probably the best suited for our purposes is to
consider the derivative
\begin{equation*}\label{derseg}
\langle E''(\bu), (\bU, \overline\bU)\rangle=\left.\frac d{d\delta}
\right|_{\delta=0} \langle E'(\bu+\delta\bU), \overline\bU\rangle,
\end{equation*}
where both vector fields $\bU$ and $\overline\bU$ belong to $\bbH$. In
our situation, and in view of (\ref{prider}), we have
\[
\begin{array}{rl}
\langle E''_0(\bu), (\bU, \overline\bU)\rangle=&\left. \displaystyle
\frac d {d\delta}\right|_{\delta=0} \displaystyle \int_0^1
(\bF^\perp(\bu+\delta\bU)\cdot(\bu'+\delta\bU'))\times\\
&\qquad\qquad[\nabla\bF^\perp(\bu+\delta\bU)\overline\bU\cdot(\bu'+\delta\bU')
+\bF^\perp(\bu+\delta\bU)\cdot\overline\bU']\,dt\\
=&\displaystyle\int_0^1(\bF^\perp(\bu)\cdot\bu')[\nabla^2\bF^\perp(\bu):(\bU,
\overline\bU, \bu')+\nabla\bF^\perp(\bu):(\overline\bU, \bU')\\
&\qquad \qquad \qquad  +
\nabla\bF^\perp(\bu):(\bU, \overline\bU')]\,dt\\
&+\displaystyle \int_0^1[\nabla\bF^\perp(\bu)\bU\cdot\bu'+
\bF^\perp(\bu)\cdot\bU']\times \\
&\qquad\qquad[\nabla\bF^\perp(\bu)\overline\bU\cdot\bu'+
\bF^\perp(\bu)\cdot\overline\bU']\,dt
\end{array}
\]
Through this long formula, we can understand, for such a $\bu$ given
and fixed, the linear operator 
$$
E''_0(\bu):\espacioo\to\espacioo. 
$$
Let 
$\bU\in\espacioo$
be given. The image 
$$
\bV=E''_0(\bu)\bU\in\espacioo
$$
of $\bU$ under the linear map $E''_0(\bu)$ is determined through the identity
\begin{equation}\label{segunder}
\langle E''_0(\bu), (\bU, \overline\bU)\rangle=\langle E''_0(\bu)\bU,
\overline\bU\rangle=\langle\bV, \overline\bU\rangle
\end{equation}
for all 
$\overline\bU\in\espacioo$.
The element $\bV$ defined through
(\ref{segunder}) is the solution of a standard quadratic variational
problem for which the weak form of its optimality condition is
precisely (\ref{segunder}). 

\subsubsection{Compactness} 
The calculations in the preceding subsection enable us to prove the following facts.


\begin{lemma}\label{compacidad}
Let $\{\bu_j\}$ be a uniformly bounded sequence in $\espacioo$, and
$\{\bv_j\}$ the sequence of derivatives
$$
\bv_j=E'_0(\bu_j)\in\espacioo
$$ 
which are solutions of
\eqref{derivadas} for $\bu=\bu_j$. Then $\{\bv_j\}$ is relatively
compact in $\espacioo$.
\end{lemma}

\begin{proof}
For the sake of notation, set
$$
\bG_j\equiv(\bF^\perp(\bu_j)\cdot\bu'_j)\bF^\perp(\bu_j),\quad \bH_j
\equiv(\bF^\perp(\bu_j)\cdot\bu'_j)\bu'^T_j D\bF^\perp(\bu_j).
$$
If $\{\bu_j\}$ is uniformly bounded in $\espacioo$, we know that a
certain subsequence (not relabelled) of $\{\bu_j\}$ converges weakly
to some $\bu$ in $\espacio$. Set
$$
\bG\equiv(\bF^\perp(\bu)\cdot\bu')\bF^\perp(\bu),\quad \bH
\equiv(\bF^\perp(\bu)\cdot\bu')\bu'^T D\bF^\perp(\bu).
$$
If we put 
$\bv_j=E'_0(\bu_j), \bv=E'_0(\bu)$,
then
\eqref{formadebil0} implies
\begin{gather}
\int_0^1[\bv''_j\cdot\bU''+(\bv'_j+\bG_j)\cdot\bU'+(\bv_j+\bH_j)\cdot\bU]\,dt=0,\nonumber\\
\int_0^1[\bv''\cdot\bU''+(\bv'+\bG)\cdot\bU'+(\bv+\bH)\cdot\bU]\,dt=0,\nonumber
\end{gather}
for every 
$\bU\in\espacioo$.
By substracting one from the other
$$
\int_0^1[(\bv''_j-\bv'')\cdot\bV''+(\bv'_j-\bv'+\bG_j-\bG)\cdot\bV'+(\bv_j-\bv+\bH_j-\bH)\cdot\bV]\,dt=0
$$
for every 
$\bU\in\espacioo$.
We can take $\bU=\bv_j-\bv$ to find that
$$
\int_0^1[|\bv''_j-\bv''|^2+(\bv'_j-\bv'+\bG_j-\bG)\cdot(\bv'_j-\bv')+(\bv_j-\bv+\bH_j-\bH)\cdot(\bv_j-\bv)]
\,dt=0.
$$
This equality can be reorganized as
$$
\|\bv_j-\bv\|^2_{\espacio}=-\int_0^1[(\bG_j-\bG)\cdot(\bv'_j-\bv')+
(\bH_j-\bH)\cdot(\bv_j-\bv)]\,dt,
$$
Hence, by the standard H\"older inequality for integrals, we can also
have
\begin{equation}\label{desig}
\|\bv_j-\bv\|_{\espacio}\le \|\bG_j-\bG\|_{L^2([0, 1]; \R^2)}+\|
\bH_j-\bH\|_{L^2([0, 1]; \R^2)}.
\end{equation}
Since the weak convergence of $\bu_j$ to $\bu$ in $\espacioo$ implies
weak convergence up to second derivatives, by the classical
Rellich--Kondrachov Theorem, which implies that the injection $W^{2,
p}\subset W^{1, p}$ is always compact (see Theorem 9.16 in
\cite{brezis} for the case with first derivatives $W^{1, p}\subset
L^p$), we conclude the convergences
$\bG_j\to\bG, \bH_j\to\bH$
 strongly in
$L^2([0, 1]; \R^2)$. The proof is then a direct consequence of
\eqref{desig}.
\end{proof}

As a direct consequence of Lemma \ref{compacidad}, we have:

\begin{corollary}\label{cc}
The map 
$$
E'_0:\espacioo\to\espacioo
$$ 
is compact.
\end{corollary}

This compactness property is the only reason why the functional $E_0$
has to be perturbed by a norm involving up to second derivatives. If
we had just perturbed $E_0$ up to first derivatives, we would not
have the strong convergence of the vector fields $\bG_j$ and $\bH_j$
to $\bG$ and $\bH$, respectively, in the proof of Lemma
\ref{compacidad}.

Another main necessary ingredient is the
compactness of the linear operator 
$$
E''_0(\bu):\espacioo\to\espacioo
$$
for each fixed $\bu$.

The form \eqref{segunder} is especially suited to show
the compactness we are after. Set
\[
\begin{array}{rl}
\overline\bF=&(\bF^\perp(\bu)\cdot\bu')\nabla^2\bF^\perp(\bu):\bu'+
\nabla\bF^\perp(\bu)\bu'\otimes\nabla\bF^\perp(\bu)\bu',\\
\overline\bG=&\nabla\bF^\perp(\bu)+\bF^\perp(\bu)\otimes\nabla
\bF^\perp(\bu)\bu',\\
\overline\bH=&\nabla\bF^\perp(\bu)+\nabla\bF^\perp(\bu)\bu'\otimes
\bF^\perp(\bu),\\
\overline\bJ=&\bF^\perp(\bu)\otimes\bF^\perp(\bu).
\end{array}
\]
This choice is dictated so that
$$
\langle E''_0(\bu), (\bU,
\overline\bU)\rangle=\int_0^1[\overline\bF(\bU,
\overline\bU)+\overline\bG(\bU', \overline\bU)+\overline\bH(\bU,
\overline\bU')+\overline\bJ(\bU', \overline\bU')]\,dt.
$$
Exactly as in Lemma \ref{compacidad}, one can show the following.
\begin{lemma}
For fixed, given 
$\bu\in\espacioo$, 
the operator
$\bU\mapsto\bV=E''_0(\bu)\bU$
is self-adjoint and compact.
\end{lemma}
\begin{proof}
Assume $\{\bU_j\}$ is bounded in $\espacioo$. In particular, and for
reasons already pointed out earlier, $\bU'_j\to\bU'$ uniformly for
some 
$\bU\in\espacioo$.
Let $\bV_j$ and $\bV$ determined through
\eqref{segunder}, respectively. Then
$$
\|\bV_j-\bV\|^2=\langle E''_0(\bu)(\bU_j-\bU), \bV_j-\bV\rangle,
$$
and
$$
\|\bV_j-\bV\|\le \|E''_0(\bu)(\bU_j-\bU)\|.
$$
The key point is to realize, in the formulas above, that in
$E''_0(\bu)(\bU_j-\bU)$ only up to first derivatives of the
differences $\bU_j-\bU$ occur, and these converge strongly to zero.
Hence 
$\|\bV_j-\bV\|\to0$.
\end{proof}

\section{Acknowledgements}
No AI tools whatsoever have been used in this work.

\end{document}